\documentclass[10pt,draft]{amsart}
\usepackage{amsmath}
\usepackage{amscd}
\usepackage{amssymb}
\usepackage[all]{xy}
\makeindex

\newtheorem{thm}[equation]{Theorem}
\newtheorem{prop}[equation]{Proposition}
\newtheorem{cor}[equation]{Corollary}
\newtheorem{lem}[equation]{Lemma}

\theoremstyle{definition}
\newtheorem{defn}[equation]{Definition}
\newtheorem{rem}[equation]{Remark}
\newtheorem{exmp}[equation]{Example}

\numberwithin{equation}{section}

\newcommand{\arr}{\rightarrow}

\newcommand{\xarr}{\xrightarrow}

\newcommand{\cat}[1]{\operatorname{\mathsf{#1}}}
\newcommand{\iso}{\stackrel{\sim}{\rightarrow}}
\newcommand{\opn}{\operatorname}
\newcommand{\inj}{\hookrightarrow}
\newcommand{\rmitem}[1]{\item[\text{\textup{(#1)}}]}

\newcommand{\mcal}[1]{\mathcal{#1}}

\newcommand{\mrm}[1]{\mathrm{#1}}

\newcommand{\Hom}{\operatorname{Hom}}
\newcommand{\Ext}{\operatorname{Ext}}

\newcommand{\Pj}{\operatorname{\mathsf{Proj}}}
\newcommand{\pj}{\operatorname{\mathsf{proj}}}
\newcommand{\Ij}{\operatorname{\mathsf{Inj}}}
\newcommand{\Mod}{\operatorname{\mathsf{Mod}}}
\newcommand{\fmod}{\operatorname{\mathsf{mod}}}
\newcommand{\add}{\operatorname{\mathsf{add}}}

\newcommand{\ten }{\otimes}

\newcommand{\lten}{{\otimes}^{\boldsymbol{L}}}

\newcommand{\LCM}{\operatorname{\mathsf{LCM}}}
\newcommand{\CM}{\operatorname{\mathsf{CM}}}
\newcommand{\Morph}{\operatorname{\mathsf{Mor}}}

\newcommand{\dd}{\mathcal{D}}

\newcommand{\z}{\mathbf{Z}}

\title[Polygon of recollements]
{Polygon of recollements and $N$-complexes}
\author{Osamu Iyama, Kiriko Kato and Jun-ichi Miyachi}
\date{\today}
\address{O. Iyama: Graduate School of Mathematics, Nagoya University Chikusa-ku, Nagoya, 464-8602
Japan}
\email{iyama@math.nagoya-u.ac.jp}
\address{K. Kato:  Graduate School of Science, Osaka Prefecture University,
1-1 Gakuen-cho, Nakaku, Sakai, Osaka 599-8531, JAPAN}
\email{kiriko@mi.s.osakafu-u.ac.jp}
\address{J. Miyachi: Department of Mathematics, Tokyo Gakugei
University, Koganei-shi, Tokyo, 184-8501, Japan}
\email{miyachi@u-gakugei.ac.jp}
\subjclass{18E30, 16G99}

\begin{document}

\begin{abstract}
We study a structure of subcategories which are called a polygon of recollements in a triangulated category.
First, we study a $2n$-gon of recollements in an $(m/n)$-Calabi-Yau triangulated category.
Second, we show the homotopy category $\cat{K}(\Morph_{N-1}(\mcal{B}))$ 
of complexes of an additive category $\Morph_{N-1}(\mcal{B})$ 
of $N-1$ sequences of split monomorphisms of an additive category $\mcal{B}$
has a $2N$-gon of recollments.
Third, we show the homotopy category $\cat{K}_{N}(\mcal{B})$ 
of $N$-complexes of $\mcal{B}$
has also a $2N$-gon of recollments.
Finally, we show there is a triangle equivalence between $\cat{K}(\Morph_{N-1}(\mcal{B}))$ and $\cat{K}_{N}(\mcal{B})$.
\end{abstract}

\maketitle

\tableofcontents

\setcounter{section}{-1}
\section{Introduction}\label{intro}

The notion of recollement of triangulated categories was introduced by Beilinson, Bernstein and Deligne in connection with derived categories of sheaves of topological spaces (\cite{BBD}).
One of the authors introduced the notion of stable $t$-structure in a triangulated category 
\cite{Mi1}, and studied relations to recollements.
Afterwards this notion was studied by many authors under a lot of names, e.g.
a torsion pair, a semiorthogonal decomposition, Bousfield localization.

\begin{defn} \label{st-tprime}
Let $\mcal{D}$ be a triangulated category with the translation functor $\Sigma$.
A pair $(\mcal{U}, \mcal{V})$ of full subcategories of $\mcal{D}$ is called a {\it stable t-structure} in $\mcal{D}$ provided that
\begin{enumerate}
\rmitem{a}  $\mcal{U}=\Sigma\mcal{U}$ and $\mcal{V}=\Sigma\mcal{V}$.
\rmitem{b}  $\opn{Hom}_{\mcal{D}}(\mcal{U}, \mcal{V}) = 0$.
\rmitem{c}  For every $X  \in \mcal{D}$, there exists a triangle $U \arr X \arr V \arr 
\Sigma U$
with $U \in \mcal{U}$ and $V \in \mcal{V}$.
\end{enumerate}
\end{defn}

In \cite{IKM}, we introduced the notion of polygons of recollements in a triangulated category,
and studied the properties of triangle of recollements in connection with
the homotopy category of unbounded complexes with bounded homologies,
its quotient category by the homotopy category of bounded complexes, and
the stable category of Cohen-Macaulay modules over an Iwanaga-Gorenstein ring.
Moreover, we studied derived categories $\cat{D}_{N}(\mcal{A})$ of $N$-complexes of an
abelian category $\mcal{A}$, and show$\cat{D}_{N}(\mcal{A})$  is a triangle equivalent to 
the derived category $\cat{D}(\Morph_{N-1}(\mcal{A}))$ of ordinary complexes over an abelian category $\Morph_{N-1}(\mcal{A})$ of $N-1$ sequences of morphisms of $\mcal{A}$ (\cite{IKM2}).
In this article, we study the properties of polygons of recollements in various categories.

\begin{defn}
Let $\mcal{D}$ be a triangulated category, and let $\mcal{U}_{1}, \cdots,  \mcal{U}_{n}$
be full triangulated subcategories of $\mcal{D}$.
An $n$-tuple $(\mcal{U}_1, \mcal{U}_2, \cdots ,  \mcal{U}_n)$ is called
an $n$-gon of recollements in $\mcal{D}$
if $(\mcal{U}_{i}, \mcal{U}_{i+1})$ is a stable t-structure in $\mcal{D}$ ($1 \leq i \leq n$), where
$\mcal{U}_1=\mcal{U}_{n+1}$.
\end{defn}

In Section \ref{t-strecoll}, we recall the notions of stable $t$-structures and recollement
polygons of recollements in a triangulated category.

\begin{prop}[Proposition \ref{p20Apr30}]
Let $\mcal{D}_1$, $\mcal{D}_2$ be triangulated categories. 
Let $( \mcal{U}_1 ,\cdots , \mcal{U}_{n})$ and
 $( \mcal{V}_1 ,\cdots , \mcal{V}_{n})$
 be $n$-gons of recollements in $\mcal{D}_1$ and $\mcal{D}_2$, respectively.
 Assume a triangle functor $F:\mcal{D}_{1} \to \mcal{D}_{2}$ 
sends $( \mcal{U}_1 , \cdots, \mcal{U}_{n} )$ to
$( \mcal{V}_1 , \cdots, \mcal{V}_{n} )$. 
If $F\mid _{\mcal{U}_t }$ and $F\mid _{\mcal{U}_{t+1}}$ 
are triangle equivalences for some $t$, so is $F$.
\end{prop} 

In Section \ref{t-strecoll}, we study stable $t$-structures with relation to relations to 
contravariantly finite categories and Calabi-Yau triangulated categories.

\begin{prop}[Proposition \ref{n-gon1}]
Let $\mcal{D}$ be an $(m/n)$-Calabi-Yau triangulated category.
For any functorially finite thick subcategory $\mcal{U}_1$ of $\mcal{D}$, we put $\mcal{U}_{i+1}:=\mcal{U}_i^\perp$ for any $i$.
Then we have an $l$-gon $(\mcal{U}_{1}, \cdots , \mcal{U}_{l})$ of recollements in $\mcal{D}$ 
for some positive divisor $l$ of $2n$.
\end{prop}

In Section \ref{n-gonII}, we constructs polygons of recollements
in the derived derived category of modules over an algebra, and them
in the stable category of Cohen-Mcaulay modules over an Iwanaga-Gorenstein ring.

\begin{thm}[Theorem \ref{n-gon3}]
Let $A$ be a finite dimensional $k$-algebra of finite global dimension such that $A/J_A$ is separable over a field $k$ and
$\cat{D}^{\mrm{b}}(\fmod A)$ is $(m/n)$-Calabi-Yau.
Let $R$ be a coherent $k$-algebra of finite self-injective dimension as both sides.
For any functorially finite thick subcategory $\mcal{U}_1$ of $\cat{D}^{\mrm{b}}(\fmod A)$, we put $\mcal{U}_{i+1}:=\mcal{U}_i^\perp$ for any $i$.
Then there is a positive divisor $l$ of $2n$ such that we have an $l$-gon 
$(\mcal{U}_{1}^{R}, \cdots , \mcal{U}_{l}^{R})$ of recollements in $\cat{D}^{\mrm{b}}(\fmod R\ten_k A)$
and an $l$-gon $(Q(\mcal{U}_{1}^{R}), \cdots , Q(\mcal{U}_{l}^{R}))$ of recollements in
 $\underline{\CM}(R\ten_k A)$
 \end{thm}

In Section \ref{Tcpx}, we the homotopy category {\small$\cat{K}(\Morph_{N-1}^{\mrm{sm}}(\mcal{B}))$}
of the category $\Morph_{N-1}^{\mrm{sm}}(\mcal{B})$ of $N-1$ sequences of split monomorphisms
in an additive category $\mcal{B}$.

\begin{thm}[Theorem \ref{th:ngon@trg01}]
Let $\mcal{B}$ be an additive category.
Then there is a $2N$-gon of recollements in $\cat{K}(\Morph^{\mrm{sm}}_{N-1}(\mcal{B}))$:
{\small
\[
(\mcal{F}^{[1,N-1]},\mcal{E}^{[2,N-1]}, \mcal{E}^{1}, \mcal{F}^{[1,2]}, \cdots,
\mcal{E}^{s}, \mcal{F}^{[s,s+1]}, \cdots, \mcal{E}^{N-2}, \mcal{F}^{[N-2,N-1]}, \mcal{E}^{N-1}, \mcal{E}^{[1,N-2]})
\]
}
\end{thm}

In Section \ref{TNcpx}, we study the homotopy category $\cat{K}_{N}(\mcal{B})$
of $N$-complexes of objects of an additive category $\mcal{B}$.

\begin{thm}[Corollary \ref{n-cpxgon04}]
We have a recollement of $\cat{K}_{N}(\mcal{B})$:
\[\xymatrix{
\cat{K}_{N-r}(\mcal{B}) \ar@<-1ex>[r]^{i_{s*}} 
& \cat{K}_{N}(\mcal{B})
\ar@/_1.5pc/[l]^{i_{s}^{*}} \ar@/^1.5pc/[l]_{i_{s}^{!}} \ar@<-1ex>[r]^{j_{s}^{*}} 
& \cat{K}_{r+1}(\mcal{B})
\ar@/_1.5pc/[l]^{j_{s!}} \ar@/^1.5pc/[l]_{j_{s*}}
}\]
\end{thm}

\begin{cor}[Corollary \ref{cpx2N-gon}]
There is a $2N$-gon of recollements in $\cat{K}_{N}(\mcal{B})$:
\[
(\mcal{F}_{1}^{N-2}, \mcal{F}_{0}^{1}, \mcal{F}_{2}^{N-2},  \mcal{F}_{1}^{1}, 
\cdots, \mcal{F}_{r+1}^{N-2}, \mcal{F}_{r}^{1}, 
\cdots, \mcal{F}_{N-1}^{N-2}, \mcal{F}_{N-2}^{1}, \mcal{F}_{0}^{N-2}, \mcal{F}_{N-1}^{1})
\]
\end{cor}

In Section \ref{TrieqDN}, we construct a triangle functor 
$\underline{F}_{N}: \cat{K}(\Morph^{\mrm{sm}}_{N-1}(\mcal{B})) \to \cat{K}_{N}(\mcal{B})$
which sends the above $2N$-gon of $\cat{K}_{N}(\mcal{B})$ to
the above $2N$-gon of $\cat{K}_{N}(\mcal{B})$.
Therefore we have the result.

\begin{thm}[Theorem \ref{KNhtp}]
Let $\mcal{B}$ be an additive category, then we have triangle equivalences:
\[
\cat{K}^{\sharp}(\Morph^{\mrm{sm}}_{N-1}(\mcal{B})) \simeq \cat{K}^{\sharp}_{N}(\mcal{B})
\]
where $\sharp=\text{nothing}, -, +, \mrm{b}$.
\end{thm}

\section{Stable t-structures and recollements}\label{t-strecoll}

We recall the notion of recollements and study their relationship with stable t-structures. 
This correspondence enables us to understand (co)localizations and recollements
by way of subcategories instead of quotient categories.  
In Proposition \ref{st-t-1second} we see that a recollement corresponds to two consecutive stable t-structures. 
First we see that a (co)localization and a stable t-structure 
essentially describe the same phenomenon, 
using the methods which are similar to ones in recollements \cite{BBD}.

Next we recall the notion of a recollement which consists of a localization and 
a colocalization. 

\begin{defn} [\cite{BBD}] 
We call a diagram
\[\xymatrix{
\mcal{D'} \ar@<-1ex>[r]^{i_{*}}
& \mcal{D}
\ar@/^1.5pc/[l]_{i^{!}} \ar@/_1.5pc/[l]^{i^{*}}  \ar@<-1ex>[r]^{j^{*}} 
& \mcal{D''}
\ar@/_1.5pc/[l]^{j_{!}} \ar@/^1.5pc/[l]_{j_{*}}
}\]
of triangulated categories and functors a \emph{recollement} if it satisfies the following:
\begin{enumerate}
\item $i_*$, $j_!$, and $j_*$ are fully faithful.
\item $(i^* , i_* )$, $(i_* , i^!)$, $(j_! , j^* )$, and $(j^* , j_* )$ are adjoint pairs. 
\item there are canonical embeddings $\opn{Im} j_! \inj \opn{Ker}i^*$, $\opn{Im}i_* \inj\opn{Ker} j^*$, and \\
$\opn{Im} j_*\inj\opn{Ker}i^!$ which are equivalences.
\end{enumerate}
\end{defn}

We remember that a recollement corresponds to a pair of 
consecutive stable t-structures.

\begin{prop} [\cite{Mi1}] \label{st-t-1second}
~
\begin{enumerate}
\item Let 
\[\xymatrix{
\mcal{D'} \ar@<-1ex>[r]^{i_{*}}
& \mcal{D}
\ar@/^1.5pc/[l]_{i^{!}} \ar@/_1.5pc/[l]^{i^{*}}  \ar@<-1ex>[r]^{j^{*}} 
& \mcal{D''}
\ar@/_1.5pc/[l]^{j_{!}} \ar@/^1.5pc/[l]_{j_{*}}
}\] be a recollement. Then 
$(\mcal{U}, \mcal{V} )$ and $(\mcal{V}, \mcal{W})$ 
are stable t-structures in $\mcal{D}$  
where we put $\mcal{U} = \opn{Im}j_{!}$, $\mcal{V} =\opn{Im}i_{*}$ and $\mcal{W} =\opn{Im}j_{*}$. 

\item  
Let $(\mcal{U}, \mcal{V})$ and $(\mcal{V}, \mcal{W})$ be stable t-structures in $\mcal{D}$.
Then for the canonical embedding $i_{*}:\mcal{V} \to \mcal{D}$,
there is a recollement 
\[\xymatrix{
\mcal{V} \ar@<-1ex>[r]^{i_{*}}
& \mcal{D}
\ar@/_1.5pc/[l]^{i^{*}} \ar@/^1.5pc/[l]_{i^{!}} \ar@<-1ex>[r]^{j^{*}} 
& \mcal{D}/\mcal{V}
\ar@/_1.5pc/[l]^{j_{!}} \ar@/^1.5pc/[l]_{j_{*}}
}\]
such that $\opn{Im}j_! = \mcal{U}$ and 
$\opn{Im}j_* =\mcal{W}$.
\end{enumerate}

In each cases, for every object $X$ of ${\mcal D}$ adjunction arrows of adjoints induce triangles 
\[ \begin{aligned} i_* i^! X \to X \to j_* j^* X \to \Sigma i_* i^! X , \\
j_! j^* X \to X \to i_* i^* X \to \Sigma j_! j^* X .  
\end{aligned} \]
\end{prop}

Thirdly, we introduce the notion of a polygon of recollements. 

\begin{defn}
Let $\mcal{D}$ be a triangulated category, and let $\mcal{U}_{1}, \cdots,  \mcal{U}_{n}$
be full triangulated subcategories of $\mcal{D}$.
We call $(\mcal{U}_1, \mcal{U}_2, \cdots ,  \mcal{U}_n)$
an $n$-gon of recollements in $\mcal{D}$
if $(\mcal{U}_{i}, \mcal{U}_{i+1})$ is a stable t-structure in $\mcal{D}$ ($1 \leq i \leq n$), where
$\mcal{U}_1=\mcal{U}_{n+1}$.
\end{defn}

An $n$-gon  of recollements results in strong symmetry, and it induces three recollements
as the name suggests.

\begin{prop}[\cite{IKM}]\label{added}
Let $\mcal{D}$ be a triangulated category.
Then $(\mcal{U}_1, \cdots, \mcal{U}_n)$ is an $n$-gon of recollements in $\mcal{D}$ if and only if there is a recollement
\[\xymatrix{
\mcal{U}_t \ar@<-1ex>[r]^{i_{t*}} 
& \mcal{D}
\ar@/_1.5pc/[l]^{i_{t}^{*}} \ar@/^1.5pc/[l]_{i_{t}^{!}} \ar@<-1ex>[r]^{j_{t}^{*}} 
& \mcal{D}/\mcal{U}_t
\ar@/_1.5pc/[l]^{j_{t!}} \ar@/^1.5pc/[l]_{j_{t*}}
}\]
such that the essential image $\opn{Im} j_{t!}$ is $\mcal{U}_{t-1}$, and that
the essential image $\opn{Im}j_{t*}$ is $\mcal{U}_{t+1}$ 
for any $t \!\! \mod \! n$.
In this case all the relevant subcategories $\mcal{U}_t$ and 
the quotient categories $\mcal{D}/\mcal{U}_t$ are triangle equivalent. 
\end{prop}

Finally we study the case that triangle functors preserve
localizations, colocalizations or recollements, etc. 

\begin{defn} Let $\mcal{D}_1$ and $\mcal{D}_2$ be triangulated categories and 
let $F:\mathcal{D}_1 \to \mathcal{D}_2$ be a triangle functor.
\begin{enumerate}
\item Let $(\mathcal{U}_n , \mathcal{V}_n)$ be a stable t-structure in $\mathcal{D}_n$ $(n=1,2)$. 
We say that \emph{$F$ sends $( \mcal{U}_1 , \mcal{V}_1 )$ to $( \mcal{U}_2 , \mcal{V}_2 )$}
if $F(\mcal{U}_1 )$ is contained in $\mcal{U}_2$ and 
$F(\mcal{V}_1 )$ is in $\mcal{V}_2$. 

\item  Let $( \mcal{U}_{in} , \cdots, \mcal{U}_{in})$  be 
an $n$-gon of recollements in $\mathcal{D}_i$ $(i=1,2)$. 
We say that \emph{$F$ sends $( \mcal{U}_{1n} , \cdots , \mcal{U}_{1n})$ to 
$( \mcal{U}_{2n} , \cdots, \mcal{U}_{2n})$}
if $F(\mcal{U}_{1k} )$ is contained in $\mcal{U}_{2k}$ for any $k$.

\end{enumerate}
\end{defn}

\begin{lem}[\cite{IKM}]\label{17-1Apr30}
If a triangle functor $F:\mcal{D}_{1} \to \mcal{D}_{2}$ 
sends a stable t-structure $( \mcal{U}_1 , \mcal{V}_1 )$ in $\mcal{D}_1$ 
to a stable t-structure $( \mcal{U}_2 , \mcal{V}_2 )$ in $\mcal{D}_2$. 
Then we have the following: 
\begin{enumerate}
\item If $F\mid _{\mcal{U}_1 }$ is full (resp., faithful), then 
$\Hom _{\mcal{D}_1}( U, X) \to \Hom _{\mcal{D}_2}( FU, FX)$ is 
surjective (resp., injective) for 
$U\in  {\mcal{U}_1}$ and $X\in  {\mcal{D}_1}$. 
\item If $F\mid _{\mcal{V}_1 }$ is full (resp., faithful), then 
$\Hom _{\mcal{D}_1}( X, V) \to \Hom _{\mcal{D}_2}( FX, FV)$ is 
surjective (resp., injective) for 
$X\in  {\mcal{D}_1}$ and $V\in  {\mcal{V}_1}$. 
\item If $F$ is full and $F\mid _{\mcal{U}_1}:\mcal{U}_1\to\mcal{U}_2$ and
$F\mid_{\mcal{V}_1}:\mcal{V}_1\to\mcal{V}_2$ are dense, 
then $F$ is dense. 
\end{enumerate}
\end{lem}

\begin{prop}\label{p20Apr30} 
Let $\mcal{D}_1$, $\mcal{D}_2$ be triangulated categories. 
Let $( \mcal{U}_1 ,\cdots , \mcal{U}_{n})$ and
 $( \mcal{V}_1 ,\cdots , \mcal{V}_{n})$
 be $n$-gons of recollements in $\mcal{D}_1$ and $\mcal{D}_2$, respectively.
 Assume a triangle functor $F:\mcal{D}_{1} \to \mcal{D}_{2}$ 
sends $( \mcal{U}_1 , \cdots, \mcal{U}_{n} )$ to
$( \mcal{V}_1 , \cdots, \mcal{V}_{n} )$. 
Then the following hold.
\begin{enumerate}
\item In the case that $n$ is odd, 
if $F\mid _{\mcal{U}_t }$ is fully faithful (equivalent) for some $t$, so is $F$. 
\item In the case that $n$ is even, 
if $F\mid _{\mcal{U}_t }$ and $F\mid _{\mcal{U}_{t+1}}$ 
are fully faithful (equivalent) for some $t$, so is $F$. 
\end{enumerate}
\end{prop} 

\begin{proof}
By Lemma \ref{17-1Apr30}.
\end{proof}

\section{Contravariantly finite subcategories and Stable $t$-structures}\label{n-gon}

In this section let $k$ be a field and $D:=\Hom_k(-,k)$.
The concept of stable $t$-structures is closely related to functorially finite subcategories \cite{AS}.
If $(\mcal{U},\mcal{V})$ is a stable $t$-structure of $\mcal{D}$, then clearly $\mcal{U}$ (resp., $\mcal{V}$) 
is a contravariantly (resp., covariantly) finite subcategory of $\mcal{D}$.
We shall show that a certain converse of this statement holds.
For a full subcategory $\mcal{U}$ of $\mcal{D}$, we put
\begin{eqnarray*}
\mcal{U}^\perp&:=&\{T\in\mcal{D}\ |\ \Hom_{\mcal{D}}(\mcal{U},T)=0\},\\
{}^\perp\mcal{U}&:=&\{T\in\mcal{D}\ |\ \Hom_{\mcal{D}}(T,\mcal{U})=0\}.
\end{eqnarray*}
Recall that an additive category is called \emph{Krull-Schmidt}
if any object is isomorphic to a finite direct sum of objects whose endomorphism rings are local.

\begin{defn}
Let $\mcal{C}$ be an additive category, and $\mcal{U}$ its full subcategory.
For an object $X$ of $\mcal{C}$, a morphism $U_X \xarr{f} X$ with $U_X \in \mcal{U}$ is called
a \emph{right $\mcal{U}$-approximation} if 
$\Hom_{\mcal{C}}(U,f):\Hom_{\mcal{C}}(U,U_X) \to \Hom_{\mcal{C}}(U,X)$ 
is surjective for any $U\in \mcal{U}$.
Moreover, a right $\mcal{U}$-approximation $U_X \xarr{f} X$ is called minimal if $g$ is an isomorphism 
whenever $g: U_X \to U_X$ satisfies $f\circ g=f$.

A full subcategory  $\mcal{U}$ is called a \emph{contravariantly finite} subcategory if
every object $X$ of $\mcal{C}$ has a right a $\mcal{U}$-approximation.
A left a $\mcal{U}$-approximation and a covariantly finite subcategory are defined dually.
$\mcal{U}$ is called a functorially finite subcategory if it is contravariantly finite and covariantly finite.
\end{defn}

\begin{prop}\label{wakamatsu}
Let $\mcal{D}$ be a Krull-Schmidt triangulated category.
For any contravariantly (resp., covariantly) finite thick subcategory $\mcal{U}$ of $\mcal{D}$, 
we have a stable $t$-structure $(\mcal{U},\mcal{U}^\perp)$ (resp., $({}^\perp\mcal{U},\mcal{U})$) in $\mcal{D}$.
\end{prop}

\begin{proof}
This is a consequence of Wakamatsu-type Lemma (see \cite[Prop. 3.6]{IY} for example).
For the convenience of the reader, we give the proof here.
Since $\mcal{D}$ is Krull-Schmidt category and $\mcal{U}$ is closed under direct summands, 
for any object $X$ of $\mcal{D}$ there is a triangle
$U_X \xarr{f} X \xarr{g} V \xarr{h} \Sigma U_X$ such that
$U_X \xarr{f} X$ is a minimal right $\mcal{U}$-approximation.
For any morphism $\alpha : U \to V$ with $U \in \mcal{U}$, by octahedral axiom we have a commutative diagram
\[\xymatrix{
\Sigma^{-1}U \ar@{=}[r]\ar[d]^{\Sigma^{-1}\alpha} & \Sigma^{-1}U \ar[d] \\
\Sigma^{-1}V \ar[r]^{-\Sigma^{-1}h} \ar[d]^{\beta} & U_X  \ar[r]^{f} \ar[d]^{\gamma} 
& X \ar[r]^{g} \ar@{=}[d] & V \ar[d]^{\Sigma\beta} \\
Z \ar[r]^{h'} \ar[d] & M  \ar[r]^{f'} \ar[d] & X \ar[r]^{g'} & \Sigma Z \\
U \ar@{=}[r] & U
}\]
where all columns and rows of 4 terms are triangles.  Then $M$ belongs to $\mcal{U}$.
Since $f$ is a right $\mcal{U}$-approximation, there is a morphism $\delta:M \to U_X$
such that $f\circ\delta=f'$.  Then there is a morphism $\epsilon: Z \to \Sigma^{-1}V$
such that $-\Sigma^{-1}h\circ\epsilon=\delta\circ h'$.
Then $\gamma$ is a split monomorphism because of the minimality of $f$.
Therefore $\beta$ is a split monomorphism, and $\Sigma^{-1}\alpha$ and $\alpha$ are zero.
Hence we have $V \in \mcal{U}^{\perp}$.
\end{proof}

\begin{defn}
Let $\mcal{D}$ be a $k$-linear triangulated category such that \\
$\dim_k\Hom_{\mcal{D}}(X,Y)<\infty$ for any $X,Y\in\mcal{D}$.
An autofunctor $S:\mcal{D}\to\mcal{D}$ is called a \emph{Serre functor} \cite{BK,RV} if there exists a functorial isomorphism
\[\Hom_{\mcal{D}}(X,Y)\simeq D\Hom_{\mcal{D}}(Y, SX)\]
for any $X,Y\in\mcal{D}$.

We say that $\mcal{D}$ is \emph{$(m/n)$-Calabi-Yau} for a positive integer $n$ and an integer $m$ if we have an isomorphism $S^n\simeq\Sigma^m$ of functors
\footnote{We notice that we can not cancel a common divisor of $m$ and $n$.}.
\end{defn}

Immediately we have the following.
\begin{lem}\label{properties of serre}
Let $\mcal{D}$ be a Krull-Schmidt triangulated category with a Serre functor $S$. Then the following hold.
\begin{itemize}
\item[(1)] $\mcal{U}^\perp={}^\perp(S\mcal{U})$ for any subcategory $\mcal{U}$ of $\mcal{D}$.
\item[(2)] If $(\mcal{U},\mcal{V})$ is a stable $t$-structure and $\mcal{U}$ is functorially finite in $\mcal{D}$,
then $(\mcal{V},S\mcal{U})$ is a stable $t$-structure and $\mcal{V}$ is functorially finite in $\mcal{D}$.
\end{itemize}
\end{lem}

\begin{proof}
(1) Immediate from the definition of Serre functor.

(2) We have ${}^\perp(S\mcal{U})=\mcal{U}^\perp=\mcal{V}$ by (1).
Since $S\mcal{U}$ is a covariantly finite subcategory of $\mcal{D}$,
we have that $(\mcal{V},S\mcal{U})$ is a stable $t$-structure by the dual of Proposition \ref{wakamatsu}.
Consequently $\mcal{V}$ is a functorially finite subcategory of $\mcal{D}$ by Proposition \ref{wakamatsu}.
\end{proof}

\begin{prop} \label{n-gon1}
Let $\mcal{D}$ be an $(m/n)$-Calabi-Yau triangulated category.
For any functorially finite thick subcategory $\mcal{U}_1$ of $\mcal{D}$, we put $\mcal{U}_{i+1}:=\mcal{U}_i^\perp$ for any $i$.
Then we have an $l$-gon $(\mcal{U}_{1}, \cdots , \mcal{U}_{l})$ of recollements in $\mcal{D}$ 
for some positive divisor $l$ of $2n$.
\end{prop}

\begin{proof}
By Propositions \ref{wakamatsu}, we have a stable $t$-structure $(\mcal{U}_1,\mcal{U}_{2})$ in $\mcal{D}$.
Using Lemma \ref{properties of serre}(2) inductively, we have a stable $t$-structure $(\mcal{U}_i,\mcal{U}_{i+1})$ in $\mcal{D}$
such that $\mcal{U}_{i+2}=S\mcal{U}_i$ for any $i$.
We have the statement because of $S^n\simeq \Sigma^m$.
\end{proof}

\section{Constructions of Stable $t$-structures}\label{n-gonII}

In this section, we investigate polygons of recollements in derived categories and in stable module categories.

For a ring $R$, we denote by $\Mod R$ (resp., $\fmod R$) the category of right 
(resp., finitely generated right) $R$-modules, and denote by $\Pj {R}$ (resp., $\Ij R$, $\pj {R}$) 
the full subcategory of 
$\cat{Mod}R$ consisting of projective (resp., injective, finitely generated projective) modules.
For right (resp., left) $R$-module $M_R$ (resp., ${}_{R}N$), we denote by 
$\opn{idim}M_R$ (resp., $\opn{idim}{}_{R}N$) the injective dimension
of $M_{R}$ (resp., ${}_{R}N$), and by 
$\opn{pdim}M_R$ (resp., $\opn{pdim}{}_{R}N$) the projective dimension
of $M_{R}$ (resp., ${}_{R}N$).
For $\mcal{A}$ be an abelian category and its addtive subcategory $\mcal{B}$.
we denote by $\cat{D}^{\mrm{b}}(\mcal{A})$ (resp., $\cat{K}^{\mrm{b}}(\mcal{B}$)
the derived category (resp., the homotopy category) of bounded complexes of objects of $\mcal{A}$ (resp., $\mcal{B}$).

\begin{defn}
We call a ring $R$ \emph{Iwanaga-Gorenstein} if it is Noetherian
with $\opn{idim}_RR<\infty$ and $\opn{idim}R_R<\infty$ \cite{Iw}.
We define the category of \emph{Cohen-Macaulay $R$-modules}
\footnote{
In the representation theory of orders  \cite{CR, A2, Y}, there is another
notion of Cohen-Macaulay modules which generalizes the classical
notion in commutative ring theory. These two concepts coincide for
Gorenstein orders.
} 
and the category of \emph{large Cohen-Macaulay $R$-modules} by
\begin{eqnarray*}
\CM R&:=&\{X\in\fmod R\ |\ \Ext^i_R(X,R)=0\ (i>0)\},\\
\LCM R&:=&\{X\in\Mod R\ |\ \Ext^i_R(X, \Pj R)=0\ (i>0)\}.
\end{eqnarray*}
\end{defn}

Then $\CM R$ forms a Frobenius category with the subcategory
$\pj R$ of projective-injective objects, and the stable
category $\underline{\CM}R$ forms a triangulated category \cite{H1}.
By \cite{IKM} there exist triangle equivalences
\[
\underline{\CM} R\simeq\cat{D}^{\mrm{b}}\!(\fmod R)/\cat{K}^{\mrm{b}}(\pj {R}),\quad
\underline{\LCM} R\simeq\cat{D}^{\mrm{b}}\!(\Mod R)/\cat{K}^{\mrm{b}}(\Pj {R}).
\]

For subcategories $\mcal{U}$ and $\mcal{V}$ of a triangulated category $\mcal{D}$, we put
\[\mcal{U}*\mcal{V}:=\{X\in\mcal{D}\ |\  U \to X \to V \to \Sigma U\ 
\text{is a triangle in}\ \mcal{D}\ (U\in \mcal{U}, V \in \mcal{V})\}.\]
By octahedral axiom, we have $(\mcal{U}*\mcal{V})*\mcal{W}=\mcal{U}*(\mcal{V}*\mcal{W})$.

\begin{lem}\label{about *}
Let $\mcal{D}$ be a triangulated subcategory, and $\mcal{U}$ and $\mcal{V}$ triangulated (resp., thick) subcategories
of $\mcal{D}$ satisfying $\Hom_{\mcal{D}}(\mcal{U},\mcal{V})=0$.
Then $\mcal{U}*\mcal{V}$ is a triangulated (resp., thick) subcategory of $\mcal{D}$.
\end{lem}

\begin{proof}
We only have to show $(\mcal{U}*\mcal{V})*(\mcal{U}*\mcal{V})\subset\mcal{U}*\mcal{V}$.
Since $\Hom_{\dd}(\mcal{U},\Sigma\mcal{V})=0$, we have $\mcal{V}*\mcal{U}=\cat{add}\{\mcal{U},\mcal{V}\}$.
Thus we have $(\mcal{U}*\mcal{V})*(\mcal{U}*\mcal{V})=\mcal{U}*(\mcal{V}*\mcal{U})*\mcal{V}=\mcal{U}*\cat{add}\{\mcal{U},\mcal{V}\}*\mcal{V}\subset(\mcal{U}*\mcal{U})*(\mcal{V}*\mcal{V})=\mcal{U}*\mcal{V}$,
where $\add\{\mcal{U},\mcal{V}\}$ is the additive subcategory of consisting of finite direct sums of objects of $\mcal{U}$
and $\mcal{V}$.
Thus $\mcal{U}*\mcal{V}$ is a triangulated subcategory of $\mcal{D}$.
If $\mcal{U}$ and $\mcal{V}$ are closed under direct summand, then so is $\mcal{U}*\mcal{V}$ (e.g. \cite[Prop. 2.1]{IY}).
Thus the assertion for thick subcategories follows.
\end{proof}

Let $A$ and $R$ be $k$-algebras.
For a subcategory $\mcal{U}$ of $\cat{D}^{\mrm{b}}(\fmod A)$,
we denote by $\mcal{U}^R$ the thick subcategory of $\cat{D}^{\mrm{b}}(\fmod R\otimes_kA)$ generated by
\[\{L\otimes_kX\ |\ L\in\cat{D}^{\mrm{b}}(\fmod R),\ X\in\mcal{U}\}.\]
The following observation gives us a lot of examples of stable $t$-structures in derived categories.

\begin{prop} \label{n-gon2}
Let $R$ be a $k$-algebra and $A$ a finite dimensional $k$-algebra such that $A/J_A$ is a separable $k$-algebra, where $J_A$ is the Jacobson radical.
For any stable $t$-structure $(\mcal{U},\mcal{V})$ in $\cat{D}^{\mrm{b}}(\fmod A)$, we have a stable $t$-structure $(\mcal{U}^R,\mcal{V}^R)$ in $\cat{D}^{\mrm{b}}(\fmod R\otimes_kA)$.
\end{prop}

\begin{proof}
Let $\mcal{D}:=\cat{D}^{\mrm{b}}(\fmod R\otimes_kA)$.
Since 
\[
\Hom_{\dd}(L\otimes_kU,M\otimes_kV)=
\Hom_{\cat{D}^{\mrm{b}}(\fmod R)}(L,M)\otimes_k\Hom_{\cat{D}^{\mrm{b}}(\fmod A)}(U,V)
\]
for any $L, M \in \cat{D}^{\mrm{b}}(\fmod R)$ and any $U, V \in \cat{D}^{\mrm{b}}(\fmod A)$,
we have $\Hom_{\dd}(\mcal{U}^R,\mcal{V}^R)=0$.

Since $\mcal{U}^R*\mcal{V}^R$ is a thick subcategory of $\mcal{D}$ by Lemma \ref{about *},
we only have to show $\mcal{U}^R*\mcal{V}^R$ contains $\fmod R\otimes_kA$.
Any $R\otimes_kA$-module $M$ is filtered by $R\otimes_kA$-modules $MJ_A^i/MJ_A^{i+1}$ which are semisimple $A$-modules.
We only have to show that any $R\otimes_kA$-module $N$ which is a semisimple $A$-modules belongs to $\mcal{U}^R*\mcal{V}^R$.
Since the map $(A/J_A)\otimes_k(A/J_A)\to A/J_A$, $x\otimes y\mapsto xy$ is a split epimorphism of $A^{\mrm{op}}\otimes_kA$-modules,
we have that the map $N\otimes_k(A/J_A)\to N$, $n\otimes y\mapsto ny$ is a split epimorphism of $R\otimes_kA$-modules.
Since $A/J_A\in\mcal{U}*\mcal{V}$, we have that $N\otimes_k(A/J_A)\in\mcal{U}^R*\mcal{V}^R$. Thus $N\in\mcal{U}^R*\mcal{V}^R$.
\end{proof}

The following result gives a criterion for a stable $t$-structure in the derived category to give a stable $t$-structure in the stable category.

\begin{lem}[\cite{IKM}] \label{st-stK}
Let $\mathcal{D}$ be a triangulated category, $\mathcal{C}$ a thick
subcategory of $\mcal{D}$, and 
$Q:\mathcal{D} \to \mathcal{D}/\mathcal{C}$ the canonical quotient \cite{Ne2}.
For a stable $t$-structure $(\mathcal{U}, \mathcal{V})$ in $\mathcal{D}$, the following are equivalent,
where $Q(\mathcal{U})$ is the full subcategory of $\mcal{D}/\mcal{C}$ consisting of objects
$Q(X)$ for $X \in \mcal{E}$.
\begin{enumerate}
\item  $(Q(\mathcal{U}), Q(\mathcal{V}))$ is a stable $t$-structure in $\mathcal{D}/\mathcal{C}$.
\item  $(\mathcal{U}\cap\mathcal{C}, \mathcal{V}\cap\mathcal{C})$ is a stable $t$-structure in 
$\mathcal{C}$.
\end{enumerate}
\end{lem}

The following example provides us a rich source of triangulated categories with Serre functors.

\begin{prop}\label{serre functor for A}
Let $A$ be a finite dimensional $k$-algebra of finite self-injective dimension as both sides.
Then the following hold.
\begin{itemize}
\item[(1)] $\cat{K}^{\mrm{b}}(\pj {A})$ has a Serre functor $\nu_A:=-\lten_{A}(DA)$.
\item[(2)] $\cat{K}^{\mrm{b}}(\pj {A})$ is $(m/n)$-Calabi-Yau if and only if $(DA)^{\lten_{A}n}\simeq\Sigma^mA$ in \\
$\cat{D}^{\mrm{b}}(\fmod A^{\mrm{op}}\otimes_k A)$.
\end{itemize}
\end{prop}

We have the following main result in this section.

\begin{thm} \label{n-gon3}
Let $A$ be a finite dimensional $k$-algebra of finite global dimension such that $A/J_A$ is separable over $k$ and
$\cat{D}^{\mrm{b}}(\fmod A)$ is $(m/n)$-Calabi-Yau.
Let $R$ be a coherent $k$-algebra of finite self-injective dimension as both sides.
For any functorially finite thick subcategory $\mcal{U}_1$ of $\cat{D}^{\mrm{b}}(\fmod A)$, we put $\mcal{U}_{i+1}:=\mcal{U}_i^\perp$ for any $i$.
Then there is a positive divisor $l$ of $2n$ such that we have an $l$-gon 
$(\mcal{U}_{1}^{R}, \cdots , \mcal{U}_{l}^{R})$ of recollements in $\cat{D}^{\mrm{b}}(\fmod R\ten_k A)$
and an $l$-gon $(Q(\mcal{U}_{1}^{R}), \cdots , Q(\mcal{U}_{l}^{R}))$ of recollements in
 $\underline{\CM}(R\ten_k A)$, where 
$Q :\cat{D}^{\mrm{b}}(\fmod R\ten_k A)\to \underline{\CM}(R\ten_k A)$ is the canonical quotient.
\end{thm}

\begin{proof}
According to Propositions \ref{n-gon1} and \ref{n-gon2}, 
we have an $n$-gon $(\mcal{U}_{1}^{R}, \mcal{U}_{2}^{R}, \cdots , \mcal{U}_{2n}^{R})$ 
of recollements in $\cat{D}^{\mrm{b}}(\fmod R\otimes_kA)$.
Since $A$ is of finite global dimension, there exists a triangle
\[U_i\to A\to U_{i+1}\to\Sigma U_i\]
with $U_i\in\mcal{U}_i\cap\cat{K}^{\mrm{b}}(\pj {A})$ and $U_{i+1}\in\mcal{U}_{i+1}\cap\cat{K}^{\mrm{b}}(\pj {A})$.
Applying $R\ten_k-$, we have a triangle
\[R\ten_kU_i\to R\ten_kA\to R\ten_kU_{i+1}\to\Sigma R\ten_kU_i\]
with $R\ten_kU_i\in\mcal{U}_{i}^{R}\cap\cat{K}^{\mrm{b}}(\pj {R\ten_kA})$ and $R\ten_kU_{i+1}\in\mcal{U}_{i+1 }^{R}\cap\cat{K}^{\mrm{b}}(\pj {R\ten_kA})$.
By Lemma \ref{st-stK}, we have a stable $t$-structure $(Q(\mcal{U}_{i}^{R}), Q(\mcal{U}_{i+1 }^{R}))$ in $\underline{\CM}(R\ten_k A)$.
\end{proof}

We have the following example of recollements by \cite[Cor. 5.11]{Mi1}.
\begin{prop}\label{recollement from idempotent}
Let $A$ be a finite dimensional $k$-algebra, and $e$ an idempotent of $A$.
Assume that $\Ext_{A}^{i}(A/AeA,A/AeA)=0$ ($i>0$), $\opn{pdim}{}_A(AeA)<\infty$ and $\opn{pdim}(AeA)_A<\infty$.
Then we have a recollement
\[\xymatrix{
\cat{D}^{\mrm{b}}(\fmod A/AeA)  \ar@<-1ex>[r]^{i_{e*}}
& \cat{D}^{\mrm{b}}(\fmod A) 
\ar@/_1.5pc/[l]^{i_{e}^{*}} \ar@/^1.5pc/[l]_{i_{e}^{!}} \ar@<-1ex>[r]^{j_{e}^{*}} 
& \cat{D}^{\mrm{b}}(\fmod eAe) 
\ar@/_1.5pc/[l]^{j_{e!}} \ar@/^1.5pc/[l]_{j_{e*}}
}\]
In particuler, $(\opn{Im}j_{e!}, \opn{Im}i_{e}^{*})$ and
$(\opn{Im}i_{e}^{*}, \opn{Im}j_{e*})$ are stable $t$-structures in 
$\cat{D}^{\mrm{b}}(\fmod A)$.
\end{prop}

\section{Recollement of $\cat{K}(\Morph_{N-1}^{\mrm{sm}}(\mcal{B}))$}\label{Tcpx}

In this section we study the properties of complexes of the category of $N-1$ sequences
of morphisms in an additive category $\mcal{B}$.
Throughout this section $\mcal{B}$ is an additive category.
Then the category $\cat{C}(\mcal{B})$ of complexes of objects of $\mcal{B}$
is a Frobenius category such that its cconflations are 
short exact sequences of which each term is a split exact sequence in $\mcal{B}$.

\begin{defn}\label{smcat}
We define the category $\Morph^{\mrm{sm}}_{N-1}(\mcal{B})$ (resp., $\Morph_{N-1}(\mcal{B})$) 
of sequences of morphisms in $\mcal{B}$ as follows.
\begin{itemize}
\item An object is a sequence of split monomorphisms (resp., morphisms) 
$X: X^{1} \xarr{\alpha_X^1} \cdots \xarr{\alpha_X^{N-2} } X^{N-1}$ in $\mcal{B}$. 
\item A morphism from $X$ to $Y$ is an $(N-1)$-tuple $f=(f^1,\cdots , f^{N-1})$ of morphisms $f^i:X^{i} \to Y^{i}$ such that
$f^{i+1}\alpha_X^{i} = \alpha_Y^{i+1} f^{i}$ for $1 \leq i \leq N-2$.
\end{itemize}
\end{defn}

We give technical tools to investigate the homotopy category
$\cat{K}(\Morph_{N-1}^{\mrm{sm}}(\mcal{B}))$.

\begin{defn}\label{comma01}
For an additive functor $G: \mcal{B} \to \mcal{B}'$ between additive categories,
let $(G\downarrow \mathbf{1}_{\mcal{B}'})$ be a comma category, that is the category of objects
$G(X) \xarr{\alpha} Y$ for $X \in \mcal{A}, Y \in \mcal{B}'$.
We denote by $(G\downarrow^{\mrm{sm}} \mathbf{1}_{\mcal{B}'})$ the subcategory of
$(G\downarrow \mathbf{1}_{\mcal{B}'})$ consisting of objects 
$G(X) \xarr{\alpha} Y$, where $\alpha$ are split monomorphisms.
\end{defn}

\begin{exmp}\label{comma02}
For $1 \leq r < N-1$, let $G:\Morph_{r}(\mcal{B}) \to \Morph_{N-r-1}(\mcal{B})$
(resp.,  $G:\Morph^{\mrm{sm}}_{r}(\mcal{B}) \to \Morph^{\mrm{sm}}_{N-r-1}(\mcal{B})$) 
be an additive functor defined by 
\[
G(X^{1} \xarr{\alpha^1} \cdots \xarr{\alpha^{r-1}} X^{r})=Y^{1} \xarr{\beta^1} \cdots \xarr{\beta^{N-r-2}} Y^{N-r-1}
\]
where $Y^{1}=\cdots =Y^{N-r-1}=X^{r}$ and $\beta^{1}=\cdots=\beta^{N-r-2}=1_{X^{r}}$.
Then the category $(G\downarrow \mathbf{1}_{\Morph_{N-r-1}(\mcal{B})})$
(resp., $(G\downarrow^{\mrm{sm}} \mathbf{1}_{\Morph_{N-r-1}(\mcal{B})})$)
is equivalent to $\Morph_{N-1}(\mcal{B})$
(resp., $\Morph^{\mrm{sm}}_{N-1}(\mcal{B})$).
\end{exmp}

\begin{lem}\label{lem:comma03}
For an additive functor $G: \mcal{B} \to \mcal{B}'$ between additive categories,
the following hold.
\begin{enumerate}
\item  Evrey complex of $(G\downarrow^{\mrm{sm}} \mathbf{1}_{\mcal{B}'})$ has the following form:
\[
G(X) \xarr{u_f} C(f)
\]
where $f: Y \to G(X)$ is a morphism of complexes of $\mcal{B}'$,
$C(f)$ is the mapping cone of $f$, and $u_f$ is the canonical morphism.
\item  If a complex $X$ of $\mcal{B}$ is homotopically trivial, then
a complex $G(X) \xarr{u_f} C(f)$ of $(G\downarrow^{\mrm{sm}} \mathbf{1}_{\mcal{B}'})$
is isomorphic to $0 \to\Sigma_{\mcal{B}'}Y$ in $\cat{K}(G\downarrow^{\mrm{sm}} \mathbf{1}_{\mcal{B}'})$, where $f: Y \to G(X)$.
\item  For a complex $G(X) \xarr{u_f} C(f)$ of  $(G\downarrow^{\mrm{sm}} \mathbf{1}_{\mcal{B}'})$,
if $C(f)$ is homotopically trivial, then $G(X) \xarr{u_f} C(f)$ is isomorphic to 
$G(X) \xarr{u_1} C(1_{G(X)})$ in in $\cat{K}(G\downarrow^{\mrm{sm}} \mathbf{1}_{\mcal{B}'})$.
\end{enumerate}
\end{lem}

\begin{proof}
(1)  It is trivial.\par \noindent
(2)  For a morphism $Y \xarr{f} G(X)$ of complexes of $\mcal{B}'$,
we have a triangle 
in $\cat{K}(G\downarrow^{\mrm{sm}} \mathbf{1}_{\mcal{B}'})$:
\[\xymatrix{
0 \ar[d] \ar[r] & G(X) \ar@{=}[d] \ar[r]^{1} & G(X) \ar[d]^{u_f} \ar[r] &0 \ar[d]\\
Y \ar[r]^{f} & G(X) \ar[r]^{u_f} & C(f) \ar[r]^{v_f} & \Sigma_{\mcal{B}'}Y
}\]
If $G(X)$ is homotopically trivial, then $G(X) \xarr{1} G(X)$ is $0$ in 
$\cat{K}(G\downarrow^{\mrm{sm}} \mathbf{1}_{\mcal{B}'})$, and hence
$G(X) \xarr{u_f} C(f)$ is isomorphic to $0 \to\Sigma_{\mcal{B}'}Y$.\par\noindent
(3)  For a morphism $Y \xarr{f} G(X)$ of complexes of $\mcal{B}'$,
we have a morphism between triangles in $\cat{K}(G\downarrow^{\mrm{sm}} \mathbf{1}_{\mcal{B}'})$:
\[
\xymatrix@!0{
& 0\ \ar@{->}[rr]\ar@{->}'[d][dd]
& & G(X)\ \ar@{->}[rr]^{1}\ar@{=}'[d][dd]
& & G(X)\ \ar@{->}[rr]^{1}\ar@{->}'[d]^{u_f}[dd]
& & 0\ \ar@{->}[dd]
\\
0\ \ar@{<-}[ur]\ar@{->}[rr]\ar@{->}[dd]
& & G(X)\ \ar@{->}[rr]^{\quad 1}\ar@{<-}[ur]^{1}\ar@{=}[dd]
& & G(X)\ \ar@{->}[rr]\ar@{<-}[ur]^{1}\ar@{->}[dd]
& & 0\ \ar@{<-}[ur]\ar@{->}[dd]
\\
& Y\ \ar@{->}'[r]^{f}[rr]
& & G(X)\ \ar@{->}'[r]^{u_f}[rr]
& & C(f)\ \ar@{->}'[r]^{v_f}[rr]
& & \Sigma Y\
\\
G(X)\ \ar@{->}[rr]^{1}\ar@{<-}[ur]^{f}
& & G(X)\ \ar@{->}[rr]\ar@{<-}[ur]^{1}
& & C(1_{G(X)})\ \ar@{->}[rr]\ar@{<-}[ur]
& & \Sigma G(X)\ \ar@{<-}[ur]}
\]
If $C(f)$ is homotopically trivial, then $f:Y \to G(X)$ is an isomorphism in $\cat{K}(\mcal{B}')$,
and therefore $(0,f):(0 \to Y) \to (0 \to G(X))$ is an isomorphism in
$\cat{K}(G\downarrow^{\mrm{sm}} \mathbf{1}_{\mcal{B}'})$.
By the above morphism between triangles, $G(X) \xarr{u_f} C(f)$ is isomorphic to 
$G(X) \xarr{u_1} C(1_{G(X)})$ in in $\cat{K}(G\downarrow^{\mrm{sm}} \mathbf{1}_{\mcal{B}'})$.
\end{proof}

\begin{defn}\label{adj02}
We define the following functors:
\[\begin{array}{ll}
D_{[s,t]}: \Morph^{\mrm{sm}}_{N-1}(\mcal{B}) \to \Morph^{\mrm{sm}}_{t-s+1}(\mcal{B}) &
(1\leq s \leq t \leq N-1) \\
E^{\Uparrow_{r}^{N-1}}:\Morph_{r}(\mcal{B}) \to \Morph_{N-1}(\mcal{B}) &
(1\leq r \leq N-2) \\
U_{N-1}:\mcal{B} \to \Morph_{N-1}(\mcal{B})
\end{array}\]
as follows.
For $X^{1} \xarr{\alpha_X^1} \cdots \xarr{\alpha_X^{N-2} } X^{N-1} \in 
\Morph^{\mrm{sm}}_{N-1}(\mcal{B})$,
\[\begin{aligned}
D_{[s, t]}(X^{1} \xarr{\alpha_X^1} \cdots \xarr{\alpha_X^{N-2} } X^{N-1}) =& \ 
Y^{1} \xarr{\alpha_Y^1} \cdots \xarr{\alpha_Y^{t-s} } Y^{t-s+1} \\
\end{aligned}\]
where $Y^{i}=X^{i+s-1}, \alpha_Y^i=\alpha_X^{i+s-1}$ $(1 \leq i \leq t-s+1)$.
We denote $D_{[s, s]}$ by $D_{[s]}$.
For $X^{1} \xarr{\alpha_X^1} \cdots \xarr{\alpha_X^{r-1} } X^r \in 
\Morph^{\mrm{sm}}_{r}(\mcal{B})$,
\[\begin{aligned}
E^{\Uparrow_{r}^{N-1}}(X^{1} \xarr{\alpha_X^1} \cdots \xarr{\alpha_X^{r-1} } X^r) =& \ 
Y^{1} \xarr{\alpha_Y^1} \cdots \xarr{\alpha_Y^{N-2} } Y^{N-1} \\
\end{aligned}\]
{\small
\[
Y^{i}=\begin{cases} 0 \ (1 \leq i < N-r) \\ X^{i-N+r+1} \ (N-r \leq i \leq N-1)\end{cases}, 
\alpha_{Y}^{i}=\begin{cases} 0 \ (1 \leq i < N-r) \\ \alpha_{X}^{i-N+r+1} \ (N-r \leq i \leq N-1)\end{cases}
\]
}
For $X \in \mcal{B}$,
\[\begin{aligned}
U_{N-1}(X) =& \ 
Y^{1} \xarr{\alpha_Y^1} \cdots \xarr{\alpha_Y^{N-2} } Y^{N-1} \\
\end{aligned}\]
where $Y^{i}=X, \alpha_Y^i=1_X$ $(1 \leq i \leq N-1)$.
Moreover, we use the same symbols for the corresponding functors
$D_{[s,t]}: \cat{K}(\Morph^{\mrm{sm}}_{N-1}(\mcal{B})) \to 
\cat{K}(\Morph^{\mrm{sm}}_{t-s+1}(\mcal{B}))$,
$E^{\Uparrow_{r}^{N-1}}:\cat{K}(\Morph_{r}(\mcal{B})) \to \cat{K}(\Morph_{N-1}(\mcal{B}))$ and
$U_{N-1}:\cat{K}(\mcal{B}) \to \cat{K}(\Morph_{N-1}(\mcal{B}))$.
\end{defn}

\begin{defn}\label{ngon@trg01}
We define the following full triangulated subcategories of \\
$\cat{K}(\Morph^{\mrm{sm}}_{N-1}(\mcal{B}))$:
\[\begin{array}{llll}
\mcal{E}^{[2,N-1]}=\opn{Ker}D_{[1]} &
\mcal{E}^{[1,N-2]}=\opn{Ker}D_{[N-1]} &
\mcal{E}^{1}=\opn{Ker}D_{[2,N-1]} \\
\mcal{E}^{s}=\opn{Ker}D_{[1,s-1]}\bigcap\opn{Ker}D_{[s+1,N-1]} &
\mcal{E}^{N-1}=\opn{Ker}D_{[1,N-2]}
\end{array}\]
For $1\leq s < t \leq N-1$,
$\mcal{F}^{[s,t]}$ is the full triangulated subcategory of $\cat{K}(\Morph^{\mrm{sm}}_{N-1}(\mcal{B}))$
consisting of objects $X^{1} \xarr{\alpha^1} \cdots \xarr{\alpha^{N-2}} X^{N-1}$
such that $\alpha^{s}=\cdots=\alpha^{t-1}=1$.
\end{defn}

Immediately, we have the following.

\begin{prop}\label{lastpiece}
The following hold.
\begin{enumerate}
\item  A functor $E^{\Uparrow_{N-2}^{N-1}}:\cat{K}(\Morph_{N-2}(\mcal{B})) \to \cat{K}(\Morph_{N-1}(\mcal{B}))$
induces a triangle equivalence between $\cat{K}(\Morph_{N-2}(\mcal{B}))$ and $\mcal{E}^{[2,N-1]}$.
\item  A functor $U_{N-1}:\cat{K}(\mcal{B}) \to \cat{K}(\Morph_{N-1}(\mcal{B}))$ induces
a triangle equivalence between $\cat{K}(\mcal{B})$ and $\mcal{F}^{[1,N-1]}$.
\end{enumerate}
\end{prop}

\begin{proof}
(1)
By Lemma \ref{lem:comma03}, every complex of $\mcal{E}^{[2,N-1]}$ is isomorphic to
some complex of the form
\[\xymatrix{
0 \ar[r] & X^2 \ar[r] & \cdots \ar[r] & X^{N-1}
}\]
Then it is easy to see that 
a triangle functor $E^{\Uparrow_{N-2}^{N-1}}:\cat{K}(\Morph_{N-2}(\mcal{B})) \to \mcal{E}^{[2,N-1]}$
is a triangle equivalence.

\par\noindent
(2)
Every complex of $\mcal{F}^{[1,N-1]}$ is of the form
\[\xymatrix{
X^1 \ar@{=}[r] & X^2 \ar@{=}[r] & \cdots \ar@{=}[r] & X^{N-1}
}\]
Then it is easy to see that 
a triangle functor $U^{N-1}:\cat{K}(\mcal{B}) \to \mcal{F}^{[1,N-1]}$
is a triangle equivalence.
\end{proof}

\begin{thm}\label{th:ngon@trg01}
Let $\mcal{B}$ be an additive category.
Then a $2N$-tuple of full subcategories
{\small
\[
(\mcal{F}^{[1,N-1]},\mcal{E}^{[2,N-1]}, \mcal{E}^{1}, \mcal{F}^{[1,2]}, \cdots,
\mcal{E}^{s}, \mcal{F}^{[s,s+1]}, \cdots, \mcal{E}^{N-2}, \mcal{F}^{[N-2,N-1]}, \mcal{E}^{N-1}, \mcal{E}^{[1,N-2]})
\]
}
is a $2N$-gon of recollements in $\cat{K}(\Morph^{\mrm{sm}}_{N-1}(\mcal{B}))$.
\end{thm}

\begin{proof}
First, we prove $\Hom_{\cat{K}(\Morph^{\mrm{sm}}_{N-1}(\mcal{B}))}(\mcal{X},\mcal{Y})=0$, where $\mcal{X}, \mcal{Y}$ 
are two successive subcategories of the above  $2N$-tuple.
By Example \ref{comma02}, Lemma \ref{lem:comma03} (2), 
any complex of $\mcal{E}^{[2,N-1]}$ is isomorphic to 
a complex $0 \to X^2 \to \cdots \to X^{N-1}$ in $\cat{K}(\Morph^{\mrm{sm}}_{N-1}(\mcal{B}))$.
Then $\Hom_{\cat{K}(\Morph^{\mrm{sm}}_{N-1}(\mcal{B}))}(\mcal{F}^{[1,N-1]},\mcal{E}^{[2,N-1]})=0$ is easy.
By Example \ref{comma02}, Lemma \ref{lem:comma03} (3),  for any object of $\mcal{E}^{1}$ there is a complex
$X$ such that it is isomorphic to an object
$
X \to \opn{C}(1_X) \to \cdots \to \opn{C}(1_X)
$
in $\cat{K}(\Morph^{\mrm{sm}}_{N-2}(\mcal{B}))$.
Since it is easy to see 
{\small
\[
\Hom_{\cat{K}(\Morph^{\mrm{sm}}_{N-1}(\mcal{B}))}(\mcal{E}^{[2,N-1]}, \mcal{E}^{1})\simeq
\Hom_{\cat{K}(\Morph^{\mrm{sm}}_{N-2}(\mcal{B}))}(D_{[2,N-1]}(\mcal{E}^{[2,N-1]}), D_{[2,N-1]}(\mcal{E}^{1}))
\]
}
and
\[
D_{[2,N-1]}(X \to \opn{C}(1_X) \to \cdots \to \opn{C}(1_X))=\opn{C}(1_X) \to \cdots \to \opn{C}(1_X)
\]
is homotopically $0$ in $\cat{K}(\Morph^{\mrm{sm}}_{N-2}(\mcal{B}))$, 
$\Hom_{\cat{K}(\Morph^{\mrm{sm}}_{N-1}(\mcal{B}))}(\mcal{E}^{[2,N-1]}, \mcal{E}^{1})=0$.
By Example \ref{comma02}, Lemma \ref{lem:comma03} (2), (3),  
we may assume any morphism from $\mcal{E}^{s}$ to $\mcal{F}^{[s,s+1]}$ is of the form
\[\xymatrix{
0 \ar[d] \ar[r] & \cdots \ar[r] & 0 \ar[d] \ar[r] & X^{s} \ar[d] \ar[r] & C(1_X)  \ar[d]\ar@{=}[r] & \cdots \ar@{=}[r] & C(1_X) \ar[d] \\
Y^{1} \ar[r] & \cdots \ar[r] & Y^{s-1} \ar[r] & Y^{s} \ar@{=}[r] & Y^{s+1} \ar[r] & \cdots \ar[r] & Y^{N-1}
}\]
It is easy that the above morphism is null homotopic, and then \\
$\Hom_{\cat{K}(\Morph^{\mrm{sm}}_{N-1}(\mcal{B}))}(\mcal{E}^{s}, \mcal{F}^{[s,s+1]})=0$.
Similarly, we may assume any morphism from $\mcal{F}^{[s,s+1]}$ to $\mcal{E}^{s+1}$ is of the form
\[\xymatrix{
X^{1} \ar[d] \ar[r] & \cdots \ar[r] & X^{s} \ar[d] \ar@{=}[r] & X^{s+1} \ar[d] \ar[r] & X^{s+2} \ar[d] \ar[r] & \cdots \ar[r] & X^{N-1} \ar[d] \\
0 \ar[r] & \cdots \ar[r] & 0 \ar[r] & Y  \ar[r] & C(1_Y)  \ar@{=}[r] & \cdots \ar@{=}[r] & C(1_Y) \\
}\]
It is easy that the above morphism is null homotopic, and then \\
$\Hom_{\cat{K}(\Morph^{\mrm{sm}}_{N-1}(\mcal{B}))}(\mcal{F}^{[s,s+1]}, \mcal{E}^{s+1})=0$.
Since any complex of $\mcal{E}^{N-1}$ is isomorphic to $0 \to \cdots  \to 0 \to X^{N-1}$, and
any complex of $\mcal{E}^{[1,N-2]}$ is isomorphic to $Y^1 \to \cdots  \to Y^{N-2} \to C(1_{Y^{N-1}})$
in $\Morph^{\mrm{sm}}_{N-1}(\mcal{B})$,
it is easy to see that $\Hom_{\cat{K}(\Morph^{\mrm{sm}}_{N-1}(\mcal{B}))}(\mcal{E}^{N-1}, \\ \mcal{E}^{[1,N-2]})=0$.
Since we may assume any morphism from $\mcal{E}^{[1,N-2]}$ to $\mcal{F}^{[1,N-1]}$ is of the form
\[\xymatrix{
X^1 \ar[d] \ar[r] & \cdots \ar[r] & X^{N-2} \ar[d] \ar[r] & C(1_{X^{N-1}}) \ar[d] \\
Y^{1} \ar@{=}[r] & \cdots \ar@{=}[r] & Y^{N-2} \ar@{=}[r] & Y^{N-1}
}\]
It is easy to see it is null homotopic, and then
$\Hom_{\cat{K}(\Morph^{\mrm{sm}}_{N-1}(\mcal{B})}(\mcal{E}^{[1,N-2]}, \mcal{F}^{[1,N-1]})=0$.
\par\noindent
Second, we prove $\cat{K}(\Morph^{\mrm{sm}}_{N-1}(\mcal{B}))=\mcal{X}*\mcal{Y}$, where $\mcal{X}, \mcal{Y}$ 
are two successive subcategories of the above  $2N$-tuple.
Let $X^{1} \xarr{\alpha^1} X^{2} \xarr{\alpha^2} \cdots  \xarr{\alpha^{N-2}} X^{N-1}$ be a complex
of $\cat{K}(\Morph^{\mrm{sm}}_{N-1}(\mcal{B}))$.
Since we have a short exact sequence of complexes:
\[\xymatrix{
X^1 \ar[d]^{1} \ar@{=}[r] & X^1 \ar[d]^{\alpha^1} \ar@{=}[r] &\cdots \ar@{=}[r]& X^{1} \ar[d]^{\alpha^{N-2}\cdots\alpha^1} \\
X^1 \ar[d] \ar[r]^{\alpha^1} & X^2 \ar[d] \ar[r]^{\alpha^2} &\cdots \ar[r]^{\alpha^{N-2}} & X^{N-1} \ar[d] \\
0 \ar[r] & \opn{Cok}\alpha^{1} \ar[r] & \cdots \ar[r] & \opn{Cok}\alpha^{N-2}\cdots\alpha^{1}
}\]
we have $\cat{K}(\Morph^{\mrm{sm}}_{N-1}(\mcal{B}))=\mcal{F}^{[1,N-1]}*\mcal{E}^{[2,N-1]}$.
Since we have a triangle in \\ $\cat{K}(\Morph^{\mrm{sm}}_{N-1}(\mcal{B}))$:
\[\xymatrix{
0 \ar[d] \ar[r]^{\alpha^1} & X^2 \ar[d]^{1} \ar[r]^{\alpha^2} &\cdots \ar[r]^{\alpha^{N-2}} & X^{N-1} \ar[d]^{1} \\
X^1 \ar[d] \ar[r]^{\alpha^1} & X^2 \ar[d] \ar[r]^{\alpha^2} &\cdots \ar[r]^{\alpha^{N-2}} & X^{N-1} \ar[d] \\
X^1 \ar[d] \ar[r] & \opn{C}(1_{X^2}) \ar[d] \ar[r] &\cdots \ar[r] & \opn{C}(1_{X^{N-1}}) \ar[d] \\
0 \ar[r]^{\Sigma \alpha^1} & \Sigma X^2 \ar[r]^{\Sigma \alpha^2} &\cdots \ar[r]^{\Sigma \alpha^{N-2}} & \Sigma X^{N-1} \\
}\]
we have $\cat{K}(\Morph^{\mrm{sm}}_{N-1}(\mcal{B}))=\mcal{E}^{[2,N-1]}*\mcal{E}^{1}$.
Since we have a triangle in $\cat{K}(\Morph^{\mrm{sm}}_{N-1}(\mcal{B}))$:
{\scriptsize
\[\xymatrix{
X^1 \ar[d]^{1} \ar[r]^{\alpha^1} & \cdots \ar[r]^{\alpha^{s-2}} & 
X^{s-1} \ar[d]^{1} \ar[r]^{\alpha^{s-1}} & X^{s} \ar[d]^{\alpha^{s}} \ar[r]^{\alpha^{s}} & 
X^{s+1} \ar[d]^{1} \ar[r]^{\alpha^{s+1}} & \cdots \ar[r]^{\alpha^{N-2}}  &X^{N-1} \ar[d]^{1} \\
X^1 \ar[d] \ar[r]^{\alpha^1} & \cdots \ar[r]^{\alpha^{s-2}}  & 
X^{s-1} \ar[d] \ar[r]^{\alpha^{s}\alpha^{s-1}} & X^{s+1} \ar[d] \ar@{=}[r] & 
X^{s+1} \ar[d] \ar[r]^{\alpha^{s+1}} & \cdots \ar[r]^{\alpha^{N-2}} &X^{N-1} \ar[d] \\
\opn{C}(1_{X^{1}}) \ar[d] \ar[r] & \cdots \ar[r] &
\opn{C}(1_{X^{s-1}}) \ar[d] \ar[r] & \opn{C}(\alpha^{s}) \ar[d] \ar[r]^{\alpha^{s+1}} & 
\opn{C}(1_{X^{s+1}}) \ar[d] \ar[r] & \cdots \ar[r] &\opn{C}(1_{X^{N-1}}) \ar[d] \\
\Sigma X^1 \ar[r]^{\Sigma \alpha^1} & \cdots \ar[r]^{\Sigma \alpha^{s-2}}  & 
\Sigma X^{s-1} \ar[r]^{\Sigma \alpha^{s-1}} & \Sigma X^{s} \ar[r]^{\Sigma \alpha^{s}} & 
\Sigma X^{s+1} \ar[r]^{\Sigma \alpha^{s+1}} & \cdots \ar[r]^{\Sigma\alpha^{N-2}}  & \Sigma X^{N-1} \\
}\]
}
we have $\cat{K}(\Morph^{\mrm{sm}}_{N-1}(\mcal{B}))=\mcal{E}^{s}*\mcal{F}^{[s,s+1]}$.
Since we have a triangle in $\cat{K}(\Morph^{\mrm{sm}}_{N-1}(\mcal{B}))$:
{\scriptsize
\[\xymatrix{
X^1 \ar[d]^{1} \ar[r] & \cdots \ar[r]^{\alpha^{s-1}}  & 
X^{s} \ar[d]^{1} \ar@{=}[r] & X^{s} \ar[d]^{\alpha^{s}} \ar[r]^{\alpha^{s+1}\alpha^{s}} & 
X^{s+2} \ar[d]^{1} \ar[r]^{\alpha^{s+2}} & \cdots \ar[r]^{\alpha^{N-2}} &X^{N-1} \ar[d]^{1} \\
X^1 \ar[d] \ar[r] & \cdots \ar[r]^{\alpha^{s-1}} & 
X^{s} \ar[d] \ar[r]^{\alpha^{s}} & X^{s+1} \ar[d] \ar[r]^{\alpha^{s+1}} & 
X^{s+2} \ar[d] \ar[r]^{\alpha^{s+2}} & \cdots \ar[r]^{\alpha^{N-2}}  &X^{N-1} \ar[d] \\
\opn{C}(1_{X^{1}}) \ar[d] \ar[r] & \cdots \ar[r] &
\opn{C}(1_{X^{s-1}}) \ar[d] \ar[r] & \opn{C}(\alpha^{s}) \ar[d] \ar[r]^{\alpha^{s+1}} & 
\opn{C}(1_{X^{s+1}}) \ar[d] \ar[r] & \cdots \ar[r] &\opn{C}(1_{X^{N-1}}) \ar[d] \\
\Sigma X^1 \ar[r]^{\Sigma \alpha^1} & \cdots \ar[r]^{\Sigma \alpha^{s-1}}  & 
\Sigma X^{s} \ar@{=}[r] & \Sigma X^{s} \ar[r]^{\Sigma \alpha^{s+1}\alpha^{s}} & 
\Sigma X^{s+2} \ar[r]^{\Sigma \alpha^{s+2}} & \cdots \ar[r]^{\Sigma\alpha^{N-2}}  & \Sigma X^{N-1} \\
}\]
}
we have $\cat{K}(\Morph^{\mrm{sm}}_{N-1}(\mcal{B}))=\mcal{F}^{[s,s+1]}*\mcal{E}^{s+1}$.
Since we have a triangle in {\small$\cat{K}(\Morph^{\mrm{sm}}_{N-1}(\mcal{B}))$:}
\[\xymatrix{
0 \ar[d] \ar[r] & \cdots \ar[r] & 0 \ar[d] \ar[r] & X^{N-1} \ar[d]^{1} \\
X^1 \ar[d]^{1} \ar[r]^{\alpha^1} & \cdots \ar[r]^{\alpha^{N-3}} & X^{N-2} \ar[d]^{1} \ar[r]^{\alpha^{N-2}} & X^{N-1} \ar[d] \\
X^1 \ar[d] \ar[r]^{\alpha^1} & \cdots \ar[r]^{\alpha^{N-3}} & X^{N-2} \ar[d] \ar[r] & \opn{C}(1_{X^{N-1}}) \ar[d] \\
0 \ar[r] & \cdots \ar[r] & 0 \ar[r] & \Sigma X^{N-1} \\
}\]
we have $\cat{K}(\Morph^{\mrm{sm}}_{N-1}(\mcal{B}))=\mcal{E}^{N-1}*\mcal{E}^{[1,N-2]}$.
Since we have a triangle in {\small$\cat{K}(\Morph^{\mrm{sm}}_{N-1}(\mcal{B}))$:}
\[\xymatrix{
X^{1} \ar[d]^{\alpha^{N-2}\cdots\alpha^{1}} \ar[r]^{\alpha^{1}} & \cdots \ar[r]^{\alpha^{N-3}}  & 
X^{N-2} \ar[d]^{\alpha^{N-2}}  \ar[r]^{\alpha^{N-2}} & X^{N-1} \ar[d]^{1} \\
X^{N-1} \ar[d] \ar@{=}[r] & \cdots \ar@{=}[r]  & 
X^{N-1} \ar[d] \ar@{=}[r] & X^{N-1} \ar[d] \\
\opn{C}(\alpha^{N-2}\cdots\alpha^{1}) \ar[d] \ar[r] & \cdots \ar[r] & 
\opn{C}(\alpha^{N-2}) \ar[d] \ar[r] & \opn{C}(1_{X^{N-1}}) \ar[d] \\
\Sigma X^{1} \ar[r]^{\Sigma \alpha^{1}} & \cdots \ar[r]^{\Sigma \alpha^{N-3}}  & 
\Sigma X^{N-2} \ar[r]^{\Sigma \alpha^{N-2}} & \Sigma X^{N-1}
}\]
we have $\cat{K}(\Morph^{\mrm{sm}}_{N-1}(\mcal{B}))=\mcal{E}^{N-1}*\mcal{F}^{[1,N-1]}$.
\end{proof}

\section{Recollement of $\cat{K}_{N}(\mcal{B})$}\label{TNcpx}

We fix a positive integer $N\ge2$.
Let $\mcal{B}$ be an additive category.
An \emph{$N$-complex} is a diagram
\[\cdots \xrightarrow{d_X^{i-1}}X^i\xrightarrow{d_X^i}X^{i+1}\xrightarrow{d_X^{i+1}}\cdots\]
with objects $X^i\in\mcal{B}$ and morphisms $d_X^i\in\Hom_{\mcal{B}}(X^i,X^{i+1})$ satisfying
\[d_{X}^{i+N-1}\cdots d_{X}^{i+1}d_{X}^{i}=0\]
for any $i\in\z$.

A \emph{morphism} between $N$-complexes is a commutative diagram
\[\begin{array}{ccccc}
\cdots \xrightarrow{d_X^{i-1}}&X^i&\xrightarrow{d_X^i}&X^{i+1}&\xrightarrow{d_X^{i+1}}\cdots\\
&\downarrow^{f^i}&&\downarrow^{f^{i+1}}&\\
\cdots \xrightarrow{d_Y^{i-1}}&Y^i&\xrightarrow{d_Y^i}&Y^{i+1}&\xrightarrow{d_Y^{i+1}}\cdots
\end{array}\]
with $f^i\in\Hom_{\mcal{B}}(X^i,Y^i)$ for any $i\in\z$.
We denote by $\cat{C}_N(\mcal{B})$ the category of $N$-complexes.
A collection $\mcal{S}_N(\mcal{B})$ of conflations is the collection of short exact sequences of $N$-complexes of which each term is a split short exact sequence in $\mcal{B}$.

\begin{prop}[\cite{IKM2}]\label{NcpxFrob}
A category $(\cat{C}_N(\mcal{B}), \mcal{S}_N(\mcal{B}))$ is a Frobenius category.
\end{prop}

\begin{defn}[\cite{IKM2}]\label{cone01}
Let $(X,d)$, $(Y,e)$ be objects and $f: Y \to X$ be a morphism in $\cat{C}_{N}(\mcal{B})$. 
Then the mapping cone $C(f)$ of $f$ is given as 
{\footnotesize 
\[ C(f)^m = X^m \oplus {\displaystyle \coprod_{i=m+1}^{m+N-1}}Y^i, 
d_{C(f)}^m = \left( \begin{array}{c|ccccc}
d&f&0&\cdots&\cdots&0\\
\hline
0&0&1&\ddots&&\vdots\\
\vdots&\vdots&\ddots&\ddots&\ddots&\vdots\\
\vdots&0&\cdots&0&1&0\\
0&0&\cdots&\cdots&0&1\\
0&-e^{\{N-1\}}&-e^{\{N-2\}}&\cdots&\cdots&-e\\
\end{array} \right)
\] 
\[(\Sigma ^{-1}  C(f) )^m \!\!\! = \!\!\! \coprod_{i=m-N+1}^{m-1} X^i \oplus Y^m , 
d_{\Sigma ^{-1} C(f)}^m = \left( \begin{array}{ccccc|c}
-d&1&0&\cdots&\cdots&0\\
-d^2&0&1&\ddots&&\vdots\\
\vdots&\vdots&&\ddots&\ddots&\vdots\\
\vdots&\vdots&\ddots&\ddots&1&0\\
-d^{\{N-1\}}&0&\cdots&\cdots&0&f\\
\hline
0&\cdots&\cdots&\cdots&0&e\\
\end{array} \right).
\]
}
Here $d^{\{N-1\}}$ means the $(N-1)$-power of $d$.
\end{defn}

The above mapping cone induces a morpism between conflations:
\[\xymatrix{
0 \ar[r] & X \ar[d]^{f} \ar[r]^{u_X} & \opn{C}(1_X) \ar[d]^{\psi_f} \ar[r]^{v_X} & \Sigma X \ar@{=}[d] \ar[r] & 0 \\
0 \ar[r] & Y \ar[r]^{u_f} & \opn{C}(f) \ar[r]^{v_f} & \Sigma X \ar[r] & 0
}\]
Let $I(X)=\opn{C}(1_X)$, then $I(X)$ is a projective-injective object in $\cat{C}_{N}(\mcal{B})$.
We call a sequence $Y \xarr{f} X \xarr{u_f} \opn{C}(1_X) \xarr{v_f} \Sigma X $ a (distinguished) triangle.

A {morphism} $f:X \to Y$of $N$-complexes
is called \emph{null-homotopic} if there exists $s^i\in\Hom_{\mathcal{B}}(X^i,Y^{i-N+1})$
such that
\[f^i=\sum_{j=1}^{N-1}d_Y^{i-1}\cdots d_Y^{i-N+j}s^{i+j-1}d_X^{i+j-2}\cdots d_X^i\]
for any $i\in\z$.
We denote by $\cat{K}_N(\mcal{B})$ the homotopy category of $N$-complexes.

\begin{thm}[\cite{IKM2}]\label{NcpxTricat}
A category $\cat{K}_N(\mcal{B})$ is a triangulated category.
\end{thm}

\begin{defn}\label{prolong01}
Let $N$ be an integer greater than 2.  For any integer $s$, we define functions 
$\iota_{s}^{(N-1)}:\z \to \z$, $\rho_{s}^{(N)}:\z \to \z$ as follows.
\[\begin{aligned}
\iota_{s}^{(N-1)}(s+i+kN) &= 
\begin{cases}\begin{aligned}
&s+k(N-1) &(i=0) \\
&s+i-1+k(N-1) &(0 < i < N)
\end{aligned}\end{cases} \\
\rho_{s}^{(N)}(s+i+k(N-1)) &= 
\begin{cases}\begin{aligned}
&s+kN &(i=0) \\
&s+i+1+kN &(0 < i < N-1)
\end{aligned}\end{cases}
\end{aligned}\]

For an $(N-1)$-complex $X=(X^{i}, d_{X}^{i})$, we define a complex $I_{s}^{(N-1)}(X)$ by
\[\begin{aligned}
I_{s}^{(N-1)}(X)^{i} &=X^{\iota_{s}^{(N-1)}(i)}, \\
d_{I_{s}^{(N-1)}(X)}^{i}&=\begin{cases}\begin{aligned}
&d^{\iota_{s}^{(N-1)}(i)} &(\iota_{s}^{(N-1)}(i) < \iota_{s}^{(N-1)}(i+1)) \\
&1 &(\iota_{s}^{(N-1)}(i) = \iota_{s}^{(N-1)}(i+1)) .
\end{aligned}\end{cases}
\end{aligned}\]

For an $N$-complex $Y=(Y^{i}, d_{Y}^{i})$, we define a complex $J_{s}^{(N)}(Y)$ by
\[\begin{aligned}
J_{s}^{(N)}(X)^{i} &=X^{\rho_{s}^{(N)}(i)}, \\
d_{J_{s}^{(N)}(X)}^{i}&=
d^{\rho_{s}^{(N)}(i+1)-1}\cdots d^{\rho_{s}^{(N)}(i)} .
\end{aligned}\]

Then $I_{s}^{(N-1)}: \cat{C}_{N-1}(\mcal{B}) \to  \cat{C}_{N}(\mcal{B})$ and 
$J_{s}^{(N)}: \cat{C}_{N}(\mcal{B}) \to  \cat{C}_{N-1}(\mcal{B})$ are functors.
\end{defn}

\begin{lem}\label{prolong02}
A functor  $I_{s}^{(N-1)}: \cat{C}_{N-1}(\mcal{B}) \to  \cat{C}_{N}(\mcal{B})$
 induces the triangle functor $\underline{I}_{s}^{(N-1)}: \cat{K}_{N-1}(\mcal{B}) \to  \cat{K}_{N}(\mcal{B})$.
\end{lem}

\begin{proof}
Let $X \xarr{f} Y$ be a morphism in $\cat{C}_{N-1}(\mcal{B})$.
Consider the mapping cone $C(I_{0}^{(N-1)}(f))$ of $I_{0}^{(N-1)}(X) \xarr{I_{0}^{(N-1)}(f)} I_{0}^{(N-1)}(Y)$, then
we have an exact sequence $0 \to I_{0}^{(N-1)}(Y) \to C(I_{0}^{(N-1)}(f)) \to \Sigma(I_{0}^{(N-1)}(X)) \to 0$ 
in $\cat{C}_{N-1}(\mcal{B})$.
Let $Z$ be an $N$-complex defined by
{\footnotesize 
\[\begin{aligned}
Z^{k}
=\begin{cases}
Y^{\iota^{(N-1)}_{0}(k)}\oplus \coprod_{i=k+1}^{k+N-2}X^{\iota^{(N-1)}_{0}(i)}\oplus X^{\iota^{(N-1)}_{0}(k)} 
\ (k \equiv 0 \opn{mod} N) \\
Y^{\iota^{(N-1)}_{0}(k)}\oplus \coprod_{i=k+1}^{k+N-1}X^{\iota^{(N-1)}_{0}(i)} \ (k \not\equiv 0 \opn{mod} N)
\end{cases} \\
d_Z^{k}
=\begin{cases}
\left(\begin{array}{ccc|c}
1 & \cdots & 0  & 0 \\ 
\vdots & \ddots & \vdots & \vdots \\
0 & \cdots & 1 & 0 \\\hline 
0 & \cdots & 0 & 0 
\end{array}\right) \ (k \equiv 0 \opn{mod} N) \\
\left(\begin{array}{ccccc|c}
e&f&0&\cdots&0& 0\\
0&0&1&\ddots&\vdots&\vdots\\
\vdots&\vdots&\vdots&\ddots&\vdots&\vdots\\
0&0&0&\cdots&1&\vdots\\
0&-d^{N-1}&d^{N-2}&\cdots&-d&0\\\hline
0&0&0&\cdots&0&0\\
\end{array}\right)  \ (k \not\equiv 0 \opn{mod} N)
\end{cases} \\
\end{aligned}\]
}
and 
$W$ an $N$-complex defined by
{\footnotesize 
\[\begin{aligned}
W^{k}
=\begin{cases}
\coprod_{i=k+1}^{k+N-2}X^{\iota^{(N-1)}_{0}(i)}\oplus X^{\iota^{(N-1)}_{0}(k)} \ (k \equiv 0 \opn{mod} N) \\
\coprod_{i=k+1}^{k+N-1}X^{\iota^{(N-1)}_{0}(i)} \ (k \not\equiv 0 \opn{mod} N)
\end{cases} \\
d_W^{k}
=\begin{cases}
\left(\begin{array}{ccc|c}
1 & \cdots & 0  & 0 \\ 
\vdots & \ddots & \vdots & \vdots \\
0 & \cdots & 1 & 0 \\\hline 
0 & \cdots & 0 & 0 
\end{array}\right) \ (k \equiv 0 \opn{mod} N) \\
\left(\begin{array}{cccc|c}
0&1&\cdots&0&0\\
\vdots&\vdots&\ddots&\vdots&\vdots\\
0&0&\cdots&1&\vdots\\
-d^{N-1}&d^{N-2}&\cdots&-d&0\\\hline
0&0&\cdots&0&0\\
\end{array}\right)  \ (k \not\equiv 0 \opn{mod} N)
\end{cases} \\
\end{aligned}\]
}
Let $g: C(I_{0}^{(N-1)}(f)) \to Z$ be a isomorphism defined by
{\tiny 
\[
\begin{array}{ll}
g^{k}= 
\left(\begin{array}{c|c|ccc}
1 & f & 0 & \cdots & 0 \\\hline 
0 & 0 & 1 & \cdots & 0 \\
\vdots & \vdots & 0 & \ddots & \vdots \\
0 & 0 & 0 & \cdots & 1 \\\hline 
0 & -1 & 0 & \cdots & 0
\end{array}\right)
\ (k \equiv 0 \opn{mod} N), \\

\left(\begin{array}{c|ccc|c}
1 & 0 & \cdots & 0 & 0 \\\hline 
0 & 1 & \cdots & 0 & 0 \\
\vdots & \vdots & \ddots & \vdots & \vdots \\
0 & 0 & \cdots & 1 & 0 \\\hline 
0 & -d^{\{N-1\}} & \cdots & -d & 1
\end{array}\right)
\ (k \equiv 1 \opn{mod} N), &

\left(\begin{array}{c|ccc|c|c}
1 & 0 & \cdots & 0 & 0 & 0 \\\hline 
0 & 1 & \cdots & 0 & 0 &0 \\
\vdots & \vdots & \ddots & \vdots & \vdots \\
0 & 0 & \cdots & 1 & 0 & 0 \\\hline 
0 & 0 & \cdots &0 & 1 & 1\\\hline
0 & 0 & \cdots &0 & 0& -1
\end{array}\right)
\ (k \equiv 2 \opn{mod} N), \\

\left(\begin{array}{ccc|c|c|c}
1 & \cdots & 0 & 0 & 0 & 0 \\
\vdots & \ddots & \vdots & \vdots& \vdots& \vdots \\
0 & \cdots & 1 & 0 & 0 & 0 \\\hline
0 & 0 & \cdots & 1 & 1 & 0 \\\hline 
0 & 0 & \cdots &0 & 0 & 1\\\hline
0 & 0 & \cdots &0 &-1& 0
\end{array}\right)
\ (k \equiv 3 \opn{mod} N), &

\hspace{54pt} \cdots, \\

\left(\begin{array}{c|c|cccc}
1 & 0 & 0 & 0 & \cdots & 0 \\\hline
0 & 1 & 1 & 0 & \cdots& 0 \\\hline
0 & 0 & 0 & 1 & \cdots & 0 \\
\vdots & \vdots & \vdots & \vdots & \ddots & \vdots \\\hline 
0 & 0 & 0 & 0 & \cdots & 1\\\hline
0 & 0 & -1 &0 &\cdots & 0
\end{array}\right)
\ (k \equiv N-1 \opn{mod} N) \\
\end{array}
\]
}
Then we have the following isomorphism between short exact sequences:
\[\xymatrix{
0 \ar[r] &I_{0}^{(N-1)}(Y) \ar@{=}[d] \ar[r]^{u} & C(I_{0}^{(N-1)}(f)) \ar[d]_{\wr}^{g} \ar[r]^{v}  & \Sigma(I_{0}^{(N-1)}(X))\ar[d]_{\wr}^{g'} \ar[r] & 0\\ 
0 \ar[r] &I_{0}^{(N-1)}(Y) \ar[r]^{u'} & Z \ar[r]^{v'} & W \ar[r] & 0\\ 
}\]
where 
{\footnotesize 
$u'= \begin{pmatrix}
1\cr 0\cr \vdots\cr 0\cr
\end{pmatrix}$, 
$v'=
\left(\begin{array}{cccc}
0 & 1 & \cdots & 0 \\
\vdots & \vdots & \ddots & \vdots \\ 
0 & 0 & \cdots & 1 \end{array}\right)
$}.
Therefore, $C(I_{0}^{(N-1)}(f))\simeq I_{0}^{(N-1)}(C(f))$ and 
$\Sigma(I_{0}^{(N-1)}(X))\simeq I_{0}^{(N-1)}(\Sigma(X))$ in $\cat{K}_N(\mcal{B})$.
\end{proof}

\begin{prop}\label{prolong03}
For a functor $J_{s}^{(N)}: \cat{C}_{N}(\mcal{B}) \to  \cat{C}_{N-1}(\mcal{B})$,
the following hold.
\begin{enumerate}
 \item
 $J_{s}^{(N)}$ induces the triangle functor $\underline{J}_{s}^{(N)}: \cat{K}_{N}(\mcal{B}) \to  \cat{K}_{N-1}(\mcal{B})$.
\item
$\underline{J}_{s}^{(N)}$ is a right adjoint of $\underline{I}_{s}^{(N-1)}$.
\item
$\underline{I}_{s+1}^{(N-1)}$ is a right adjoint of $\underline{J}_{s}^{(N)}$.
\item
The adjunction arrow $\mathbf{1}\to \underline{J}_{s}^{(N)}\underline{I}_{s}^{(N-1)}$ is an isomorphism.
\item
The adjunction arrow $\underline{J}_{s}^{(N)}\underline{I}_{s+1}^{(N-1)}\to \mathbf{1}$ is an isomorphism.
\end{enumerate}
\end{prop}

\begin{proof}
For $X, Z \in \cat{C}_{N-1}(\mcal{B}), Y \in \cat{C}_{N}(\mcal{B})$,
consider the following morphisms in $\cat{C}_{N}(\mcal{B})$:
\[\xymatrix{ 
I_s^{(N-1)}(X) : \ar[d] & \cdots \ar[r] & X^{s} \ar@{=}[r] \ar[d]^{f^s} & X^{s} \ar[r]^{d^s} \ar[d]^{d^sf^s}  & X^{s+1} \ar[r] \ar[d]^{f^{s+1}} & \cdots \\
Y :  \ar[d] : & \cdots \ar[r] & Y^s \ar[r]^{d^s} \ar[d]^{g^s} & Y^{s+1} \ar[r]^{d^{s+1}} \ar[d]^{g^{s+1}}  & Y^{s+2} \ar[r] \ar[d]^{g^{s+2}} & \cdots\\
I_{s+1}^{(N-1)}(Z) : & \cdots \ar[r] & Z^s \ar[r]^{d^s} & Z^{s+1} \ar@{=}[r]  & Z^{s+1} \ar[r] & \cdots
} \] 
Then these morphisms correspond to the following morphisms in $\cat{C}_{N-1}(\mcal{B})$: 
\[\xymatrix{ 
X : \ar[d] & \cdots \ar[r] & X^{s} \ar[r]^{d^s} \ar[d]^{f^s} & X^{s+1} \ar[r] \ar[d]^{f^{s+1}} & \cdots \\
J_{s}^{(N)}(Y) :  \ar[d] : & \cdots \ar[r] & Y^s \ar[r]^{d^{s+1}d^s} \ar[d]^{g^s} & Y^{s+2} \ar[r] \ar[d]^{g^{s+2}} & \cdots\\
Z : & \cdots \ar[r] & Z^s \ar[r]^{d^s} & Z^{s+1} \ar[r] & \cdots
} \] 
Then it is easy to see that ${J}_{s}^{(N)}$ is a right adjoint of ${I}_{s}^{(N-1)}$, that
${I}_{s+1}^{(N-1)}$ is a right adjoint of ${J}_{s}^{(N)}$, and that
the adjunction arrows $\mathbf{1}\to {J}_{s}^{(N)}{I}_{s}^{(N-1)}$ and
${J}_{s}^{(N)}{I}_{s+1}^{(N-1)}\to \mathbf{1}$ are isomorphisms.
By Lemmas \ref{prolong02} and Proposition \ref{triadjoint02}, we have statements.
\end{proof}

\begin{defn}\label{updown01}
We denote 
$\underline{I}_{s}^{\Uparrow_{N-r}^{N}}=
\underline{I}_{s}^{(N-1)}\cdots\underline{I}_{s}^{(N-r)}:
\cat{K}_{N-r}(\mcal{B}) \to \cat{K}_{N}(\mcal{B})$ and
$\underline{J}_{s}^{\Downarrow_{N-r}^{N}}=
\underline{J}_{s}^{(N-r+1)} \cdots \underline{J}_{s}^{(N)}:
\cat{K}_{N}(\mcal{B}) \to \cat{K}_{N-r}(\mcal{B})$.
For $1\leq r < N$, we define the full subcategory of  $\cat{K}_{N}(\mcal{B})$
\[
\mcal{F}_{s}^{r}=\{(X^i, d^i) \in \cat{K}_{N}(\mcal{B}) \mid d^{i+k}=1_{X^{i+k}} (0 \leq k < r)
\text{for}\  i\equiv s \mod N\}
\]
\end{defn}

\begin{cor}\label{updown02}
$\underline{I}_{s}^{\Uparrow_{N-r}^{N}}$, $\underline{J}_{s}^{\Downarrow_{N-r}^{N}}$
are triangle functors such that 
$\underline{I}_{s}^{\Uparrow_{N-r}^{N}}$ is a left (reps., right) adjoint of
$\underline{J}_{s}^{\Downarrow_{N-r}^{N}}$ 
(resp., $\underline{J}_{s+1}^{\Downarrow_{N-r}^{N}}$),
and that $\mcal{F}_{s}^{r}=\cat{Im}\underline{I}_{s}^{\Uparrow_{N-r}^{N}}$ is 
a full triangulated subcategory of $\cat{K}_{N}(\mcal{B})$.
\end{cor}
\begin{proof}
By Lemma \ref{prolong02}, Proposition \ref{prolong03}.
\end{proof}

\begin{thm}\label{tstNcpx01}
For $1\leq r < N$, $(\mcal{F}_{s}^{r}, \mcal{F}_{r+s+1}^{N-r-1})$ is a stable $t$-structure
in $\cat{K}_{N}(\mcal{B})$.
\end{thm}

\begin{proof}
We may assume $s=0$.
It is easy to see that 
$\Hom_{\cat{K}_{N}(\mcal{B})}(\mcal{F}_{0}^{r}, \mcal{F}_{r+1}^{N-r-1})=0$.
Let $X$ be an $N$-complex.
By Proposition \ref{prolong03}, the adjunction arrow induce a triangle
\[
\underline{I}_{0}^{\Uparrow_{N-r}^{N}}
\underline{J}_{0}^{\Downarrow_{N-r}^{N}}(X) \xarr{\varepsilon_{X}} X \xarr{u} \opn{C}(\varepsilon_{X}) \xarr{v}
\Sigma\underline{I}_{0}^{\Uparrow_{N-r}^{N}}
\underline{J}_{0}^{\Downarrow_{N-r}^{N}}(X) .
\]
Let $\sigma:=\iota_{0}^{(N-1)}\cdots\iota^{(N-r)}_{0}\rho_{0}^{(N-r+1)}\cdots\rho_{0}^{(N)}$,
and $V$ be an $N$-complex defined by
{\footnotesize 
\[\begin{aligned}
V^{k}
&=\begin{cases}
\coprod_{i=k+1}^{k+N-1}X^{\sigma(i)} \
 (k\equiv 0, \cdots , r \opn{mod} N)\\
X^{k}\oplus \coprod_{i=k+1, i\not\equiv 0 \opn{mod} N}^{k+N-1}X^{\sigma(i)} \
\ (k\equiv r+1, \cdots , N-1 \opn{mod} N) 
\end{cases}
\\
d_{V}^{k}
&=\begin{cases}
\left(\begin{array}{ccc}
1 & \cdots & 0 \\
\vdots & \ddots & \vdots \\
0 & \cdots & 1\end{array}\right)
\ (k \equiv r \opn{mod} N) \\
\left(\begin{array}{c|ccc}
0 & 1  & \cdots & 0 \\ 
\vdots & \vdots & \ddots & \vdots \\
0 & 0 & \cdots  & 1 \\\hline 
0 & 0 & \cdots  & 0 
\end{array}\right) \ (\text{othewise}) \\
\end{cases} \\
\end{aligned}\]
}
Let $h: V \to C(\epsilon_{X})$ be a monomorphism defined by
{\scriptsize 
\[\begin{array}{ll}
h^{k}=
\left(\begin{array}{cccccc}
0  & \cdots & 0& 0   \\\hline
1  & \cdots & 0& 0   \\
-d  & \ddots & \vdots& \vdots   \\
\vdots  & \ddots & 1&0   \\
0  & \cdots & -d &1
\end{array}\right)
& (k \equiv r \opn{mod} N), 
\end{array}
\]
\[\begin{array}{lll} 
& \hspace{72pt} N-r\\
N-r
&
\left(\begin{array}{cccccccc}
1  & \cdots & 0& 0 & 0 & \cdots & 0   \\
-d  & \ddots & \vdots & \vdots & \vdots& \cdots & \vdots    \\
\vdots  & \ddots & 1& 0 & 0 & \cdots & 0    \\
0  & \cdots & -d &1 & 0 & \cdots & 0 \\
0  & \cdots & 0 &-1&1 & \cdots & 0 \\
0  & \cdots & 0 &\vdots &\ddots & \ddots & \vdots \\
0  & \cdots & 0 &0 &\cdots & -1 &1 \\
0  & \cdots & 0 &0 &\cdots & 0 &-1
\end{array}\right)
& (k \equiv r+1 \opn{mod} N)
\end{array}
\]
\[\begin{array}{lll} 
& \hspace{72pt} N-r\\
N-r
&
\left(\begin{array}{ccccccc|c}
1  & \cdots & 0& 0 & 0 & \cdots & 0 &0  \\
-d  & \ddots & \vdots & \vdots & \vdots& \cdots & \vdots & \vdots  \\
\vdots  & \ddots & 1& 0 & 0 & \cdots & 0 & 0     \\
0  & \cdots & -d &1 & 0 & \cdots & 0 & 0 \\
0  & \cdots & 0 &-1&1 & \cdots & 0 & 0 \\
0  & \cdots & 0 &\vdots &\ddots & \ddots & \vdots & \vdots \\
0  & \cdots & 0 &0 &\cdots & -1 &1 &0\\
0  & \cdots & 0 &0 &\cdots & 0 &-1&0 \\\hline
0  & \cdots & 0 &0 &\cdots & 0 &0 &1
\end{array}\right)
& (k \equiv r+2 \opn{mod} N),\cdots
\end{array}\]
\[\begin{array}{lll} 
& \hspace{96pt} r+1\\
r+1
&\left(\begin{array}{cccc|ccccc}
1  & 0 & \cdots & 0 & 0  & \cdots & 0 & 0   \\
-1 & 1 & \cdots& 0 & 0  & \cdots & 0& 0   \\
0 & \ddots & \ddots & \vdots & \vdots  & \ddots & \vdots& \vdots   \\
\vdots & \ddots  & \ddots & 1 & 0  &\cdots & 0& 0   \\
0 & \cdots & 0 & -1 & 0  & \cdots & 0& 0   \\\hline
0 & \cdots & \cdots & 0 & 1  & \cdots & 0& 0   \\
0 & \cdots & \cdots & 0 & -d  & \ddots & \vdots& \vdots   \\
\vdots & \vdots & \vdots & \vdots & \vdots  & \ddots & 1& 0   \\
0 & \cdots & \cdots & 0 & 0  & \cdots & -d & 1
\end{array}\right)
& (k \equiv 0 \opn{mod} N)
\end{array}\]
\[\begin{array}{lll} 
& \hspace{96pt} r\\
r
&\left(\begin{array}{cccc|cccc}
d  & 0 & \cdots & 0 & 0  & \cdots & 0& 0   \\
-1 & 1 & \cdots& 0 & 0  & \cdots & 0& 0   \\
0 & -1 & \ddots & \vdots & \vdots  & \ddots & \vdots& \vdots   \\
\vdots & \ddots  & \ddots & 1 & 0  &\cdots & 0& 0   \\
0 & \cdots & 0 & -1 & 0  & \cdots & 0& 0   \\\hline
0 & \cdots & \cdots & 0 & 1  & \cdots & 0 & 0   \\
0 & \cdots & \cdots & 0 & -d  & \ddots & \vdots & \vdots   \\
\vdots & \vdots & \vdots & \vdots & \vdots  & \ddots & 1 & 0  \\
0 & \cdots & \cdots & 0 & 0  & \cdots & -d & 1
\end{array}\right)
& (k \equiv 1 \opn{mod} N), \cdots 
\end{array}\]

\[\begin{array}{ll} 
\left(\begin{array}{c|cccccc}
d^{r-1}   &  0  & \cdots & 0& 0   \\
-1  &  0  & \cdots & 0& 0   \\\hline
0  & 1  & \cdots & 0& 0   \\
0  & -d  & \ddots & \vdots& \vdots   \\
\vdots  & \vdots  & \ddots & 1&0   \\
0  &  0  & \cdots & -d &1
\end{array}\right)
& (k \equiv r-1 \opn{mod} N)
\end{array}\]
}

Let $p: C(\epsilon_{X}) \to \underline{I}_{r+1}^{\Uparrow_{r+1}^{N}}
\underline{J}_{r}^{\Downarrow_{r+1}^{N}}(X)$ be a epimorphism defined by
{\scriptsize
\[\begin{array}{ll}
p^{k}=
& \left(\begin{array}{cccccc}
1  & 0 & \cdots & 0
\end{array}\right)
\ (k \equiv r \opn{mod} N),
\\

& \left(\begin{array}{cccccccc}
d^{N-r-1}  & d^{N-r-2}& \cdots & d & 1& \cdots & 1& 1
\end{array}\right)
\ (k \equiv r+1 \opn{mod} N), 
\\

& \left(\begin{array}{cccccccc}
d^{N-r-2}  & d^{N-r-3}& \cdots & d & 1& \cdots & 1& 0
\end{array}\right)
\ (k \equiv r+2 \opn{mod} N), 
\\

& \hspace{120pt} \vdots \\

& \begin{array}{l} 
\hspace{54pt} r+1\\
\left(\begin{array}{cccccccc}
1  & 1& \cdots & 1 & 0& \cdots & 0
\end{array}\right)
\ (k \equiv 0 \opn{mod} N), 
\end{array} 
\\

&  \begin{array}{l} 
\hspace{60pt} r\\
\left(\begin{array}{cccccccc}
1  & d& \cdots & d & 0& \cdots & 0 
\end{array}\right)
\ (k \equiv 1 \opn{mod} N), 
\end{array} 
\\

&  \begin{array}{l} 
\hspace{60pt} r-1\\
\left(\begin{array}{cccccccc}
1  & d^2 & \cdots & d^2 & 0& \cdots & 0
& \end{array}\right)
\ (k \equiv 2 \opn{mod} N), 
\end{array} 
\\

& \hspace{120pt} \vdots \\

& \left(\begin{array}{cccccccc}
1  & d^{r-1} & 0 \cdots &  0
\end{array}\right)
\ (k \equiv r-1 \opn{mod} N).
\\

\end{array}
\]
}
Then we have the following conflations.
\[
0 \to V \xarr{h} \opn{C}(\varepsilon_{X}) \xarr{p} 
\underline{I}_{r+1}^{\Uparrow_{r+1}^{N}}
\underline{J}_{r}^{\Downarrow_{r+1}^{N}}(X) \to 0
\]
Since $V$ is a projective object in $\cat{C}_{N}(\mcal{B})$, we have an isomorphism
in $\cat{K}_{N}(\mcal{B})$
\[
\opn{C}(\varepsilon_{X}) \simeq  \underline{I}_{r+1}^{\Uparrow_{r+1}^{N}}
\underline{J}_{r}^{\Downarrow_{r+1}^{N}}(X) 
\] 
Therefore, we have a triangle $U \to X \to V \to \Sigma X$ such that $U \in \mcal{F}_{s}^{r}$, 
$V \in \mcal{F}_{r+s+1}^{N-r-1}$.\end{proof}

\begin{cor}\label{n-cpxgon04}
We have a recollement of $\cat{K}_{N}(\mcal{B})$:
\[\xymatrix{
\cat{K}_{N-r}(\mcal{B}) \ar@<-1ex>[r]^{i_{s*}} 
& \cat{K}_{N}(\mcal{B})
\ar@/_1.5pc/[l]^{i_{s}^{*}} \ar@/^1.5pc/[l]_{i_{s}^{!}} \ar@<-1ex>[r]^{j_{s}^{*}} 
& \cat{K}_{r+1}(\mcal{B})
\ar@/_1.5pc/[l]^{j_{s!}} \ar@/^1.5pc/[l]_{j_{s*}}
}\]
where 
$i_{s}^{*}=\underline{J}_{s-1}^{\Downarrow_{N-r}^{N}}$,
$i_{s*}=\underline{I}_{s}^{\Uparrow_{N-r}^{N}}$,
$i_{s}^{!}=\underline{J}_{s}^{\Downarrow_{N-r}^{N}}$,
$j_{s!}=\underline{I}_{r+s}^{\Uparrow_{r+1}^{N}}$,
$j_{s}^{*}=\underline{J}_{r+s}^{\Downarrow_{r+1}^{N}}$ and
$j_{s!}=\underline{I}_{r+s+1}^{\Downarrow_{r+1}^{N}}$.
\end{cor}

\begin{proof}
By Theorem \ref{tstNcpx01}, 
$(\mcal{F}_{r+s-N}^{N-r-1}, \mcal{F}_{s}^{r})$ and
$(\mcal{F}_{s}^{r}, \mcal{F}_{r+s+1}^{N-r-1})$ are stable $t$-structures in
$\cat{K}_{N}(\mcal{B})$.
By the proof of Theorem \ref{tstNcpx01}, the adjunction arrows induce
triangles in $\cat{K}(\mcal{B})$.
Hence we have the statement.
\end{proof}

\begin{cor}\label{cpx2N-gon}
For any integer $s$, 
\[
(\mcal{F}_{s+1}^{N-2}, \mcal{F}_{s}^{1}, \mcal{F}_{s+2}^{N-2},  \mcal{F}_{s+1}^{1}, 
\cdots, \mcal{F}_{s+r+1}^{N-2}, \mcal{F}_{s+r}^{1}, 
\cdots, \mcal{F}_{s+N-1}^{N-2}, \mcal{F}_{s+N-2}^{1}, \mcal{F}_{s}^{N-2}, \mcal{F}_{s+N-1}^{1})
\]
is a $2N$-gon of recollements in $\cat{K}_{N}(\mcal{B})$.
\end{cor}

\section{Triangle equivalence between homotopy categories}\label{TrieqDN}

\[
\mu^{s}_{r}C: X^{s-r+1} \xarr{d^{s-r+1}} \cdots \xarr{d^{s-2}} X^{s-1} \xarr{d^{s-1}} X^{s}
\] 
be an $N$-complex satisfying that 
$
X^{s-i}=C  \ (0 \leq i \leq r-1), \quad d^{s-i}=1_{C}  \ (0 < i \leq r-1).
$

\begin{lem}\label{smcatcp}
Let $\mcal{B}$ be an additive category.
Consider $\Morph^{\mrm{sm}}_{N-1}(\mcal{B})$ as a subcategory
of $\cat{C}_{N}(\mcal{B})$, then the following hold.
\begin{enumerate}
\item 
For every object $X$ of $\Morph^{\mrm{sm}}_{N-1}(\mcal{B})$,
there are objects $C_1, \cdots ,C_{N-1}$ of $\mcal{B}$ such that
$X \simeq \coprod_{i=1}^{N-1}\mu^{N-1}_{i}C_{i}$.
\item 
For any $X, Y \in \Morph^{\mrm{sm}}_{N-1}(\mcal{B})$, we have isomorphisms
\[\begin{aligned}
\Hom_{\Morph^{\mrm{sm}}_{N-1}(\mcal{B})}(X,Y) &=\Hom_{\cat{C}_{N}(\mcal{B})}(X,Y) \\
&=\Hom_{\cat{K}_{N}(\mcal{B})}(X,Y) .
\end{aligned}\]
\item 
For any $X, Y \in \Morph^{\mrm{sm}}_{N-1}(\mcal{B})$, we have
$\Hom_{\cat{K}_{N}(\mcal{B})}(X,\Sigma^{i}Y) = 0 \ (i\not=0)$.
\end{enumerate}
\end{lem}

\begin{proof}
(1)  It is trivial.\par\noindent
(2)  It is easy because the term-length of objects of $\Morph^{\mrm{sm}}_{N-1}(\mcal{B})$ is less than $N$.
\par\noindent
(3)  For every $C \in \mcal{B}$, $1 \leq r \leq N-1$ and any integer $i$, the canonical injection and projection induce an conflation in 
$\cat{C}(\mcal{B})$
\[\begin{aligned}
0 \to \mu_{r}^{iN+N-1}C \to \mu_{N}^{iN+N-1}C \to \mu_{N-r}^{(i+1)N-r-1}C \to 0 .
\end{aligned}\]
Since $\mu_{N}^{iN+N-1}C$ 
is a projective-injective objet in 
$\cat{C}_{N}(\mcal{B})$,
we have isomorphisms in $\cat{K}_{N}(\mcal{B})$:
\[
\Sigma^{j}\mu_{r}^{N-1}C\simeq\begin{cases}
\mu_{N-r}^{(1-j)N/2-r-1}C \ (j\equiv 1 \opn{mod} 2)\\
\mu_{r}^{(2-j)N/2-1}C \ (j\equiv 0 \opn{mod} 2)
\end{cases}
\]
For every $C, C' \in \mcal{B}$ and any $1 \leq r,r' \leq N-1$, we have
\[
\Hom_{\cat{K}(\mcal{B})}(\mu_{r}^{N-1}C,\Sigma^{j}\mu_{r'}^{N-1}C') = 0 \ (j\not=0) .
\]
By (1), we have the statement.
\end{proof}

\begin{defn}\label{simpleF}
For every $C \in \mcal{B}$, let
\[
\Xi^{j}\mu_{r}^{N-1}C=\begin{cases}
\mu_{N-r}^{(1-j)N/2+N-r-1}C \ (j\equiv 1 \opn{mod} 2)\\
\mu_{r}^{(2-j)N/2-1}C \ (j\equiv 0 \opn{mod} 2)
\end{cases}\]
By the proof of Lemma \ref{smcatcp},
for every $M \in \Morph^{\mrm{sm}}_{N-1}(\mcal{B})$ and any $i \in \mathbb{Z}$, there exist
the projective-injective object $I'(\Xi^iM)$ and
a functorial conflation in $\cat{C}_N(\mcal{B})$ 
\[
0 \arr \Xi^{i}M \xarr{u'_{\Xi^iM}} I'(\Xi^{i} M) \xarr{v'_{\Xi^iM}} \Xi^{i+1}M \arr   0
\]
that is, every morphism $f: M \to N$ in $\Morph^{\mrm{sm}}_{N-1}(\mcal{B})$
uniquely determines morphisms $I'(f):I'(M) \to I'(N)$ and $\Xi^if:\Xi^iM \to \Xi^iN$
which have the following commutative diagram:
\[\xymatrix{
0 \ar[r] & \Xi^{i}M \ar[r] \ar[d]^{\Xi^if} & I'(\Xi^{i} M) \ar[r]  \ar[d]^{I'(\Xi^if)} & \Xi^{i+1}M \ar[r] 
\ar[d]^{\Xi^{i+1}f} & 0 \\
0 \ar[r] & \Xi^{i}N \ar[r] & I'(\Xi^{i} N) \ar[r] & \Xi^{i+1}N \ar[r] & 0 
}\]
\end{defn}

For any (ordinary and $N$) complex $X = (X^{i}, d^{i})$, we define 
the following truncations:
\[\begin{aligned}
{\tau}_{\leq n}X & : \cdots \arr X^{n-2} \arr X^{n-1}
\arr X^{n} \arr 0 \arr \cdots ,\\
{\tau }_{\geq n}X & : \cdots \arr 0 \arr X^{n}
\arr X^{n+1}\arr X^{n+2}\arr \cdots, \\
{\tau}_{[m,n]}X & : \cdots \arr X^{m} \arr \cdots \arr X^{n} \arr 0 \arr \cdots ,\\
{\tau }_{n}X & : \cdots \arr 0 \arr X^{n} \arr 0 \arr \cdots .
\end{aligned}\]

\begin{lem}\label{prolong}
Let $\mcal{B}$ be an additive category.
There exists an exact functor 
$F_{N}: \cat{C}(\Morph^{\mrm{sm}}_{N-1}(\mcal{B})) \to \cat{C}_{N}(\mcal{B})$
which sends $\Morph^{\mrm{sm}}_{N-1}(\mcal{B})$ to $\Morph^{\mrm{sm}}_{N-1}(\mcal{B})$ as a subcategory
of $\cat{K}_{N}(\mcal{B})$
such that $F_{N}$ induces a triangle functor 
$\underline{F}_{N}: \cat{K}(\Morph^{\mrm{sm}}_{N-1}(\mcal{B})) \to \cat{K}_{N}(\mcal{B})$.
\end{lem}

\begin{proof}
Let $T: \cat{C}(\Morph^{\mrm{sm}}_{N-1}(\mcal{B})) \to \cat{C}(\Morph^{\mrm{sm}}_{N-1}(\mcal{B}))$
 be a translation functor. 
We will prove the statement by the following steps:

\par\noindent
Step 1.  We have a functor $F_{N}: \coprod_{i}T^{i}\mcal{B} \to \cat{C}_{N}(\mcal{B})$ which preserves split exact sequences.

\par\noindent
For $T^{i} X \in T^{i}\Morph^{\mrm{sm}}_{N-1}(\mcal{B})$, let $F_{N}(T^{i} X)=\Xi^{i}X$, then 
$F_{N}: \coprod_{i}T^{i}\Morph^{\mrm{sm}}_{N-1}(\mcal{B}) \to \cat{K}_{N}(\mcal{B})$ which preserves split exact sequences.

\par\noindent
Step 2.  We have a functor $F_{N}: \cat{C}^{\mrm{b}}(\Morph^{\mrm{sm}}_{N-1}(\mcal{B})) \to \cat{C}_{N}(\mcal{B})$ which preserves conflations.

\par\noindent
For $i<j$, let $X:X^{i} \xarr{d^i} X^{i+1} \to \cdots \to X^{j} \in \cat{C}^{\mrm{b}}(\Morph^{\mrm{sm}}_{N-1}(\mcal{B}))$.
By induction on $j-i$, we comstruct a functor $F_{N}$.
For $X$, we have a commutative diagram
in $\cat{C}^{\mrm{b}}(\mcal{B})$:
\[\xymatrix{
0 \ar[r] & T^{-i}X^{i} \ar[r] \ar[d]^{d^i} & I(T^{-i}X^{i}) \ar[r] \ar[d] & T^{-i-1}X^i \ar[r] \ar@{=}[d] & 0 \\
0 \ar[r] & \tau_{\geq i+1}X \ar[r] &  \tau_{\geq i}X \ar[r] & T^{-i-1}X^i \ar[r] & 0
}
\]
where all rows are exact.  Then we have a commutative diagram
in $\cat{C}_{N}(\mcal{B})$:
\begin{equation}\label{limit01}
\xymatrix@1{
0 \ar[r] & \ar @{} [dr] |{(A)} F_{N}(T^{-i}X^{i}) \ar[r] \ar[d]^{F_{N}(d^i)} & I'(F_{N}(T^{-i}X^{i})) \ar[r] \ar[d]^{\gamma_{i}} & F_{N}(T^{-i-1}X^i) \ar[r] \ar@{=}[d] & 0 \\
0 \ar[r] & F_{N}(\tau_{\geq i+1}X) \ar[r] & C(F_{N}(d^{i})) \ar[r] & F_{N}(T^{-i-1}X^i) \ar[r] & 0
}
\end{equation}
where all rows are conflations.  We take $F_{N}(\tau_{\geq i}X)=C(F_{N}(d^i))$.  
By successive step, we have a functor $F_{N}: \cat{C}^{\mrm{b}}(\Morph^{\mrm{sm}}_{N-1}(\mcal{B})) \to \cat{C}_{N}(\mcal{B})$.
For any inflation (resp., deflation) $f:X \to Y$ in $\cat{C}^{\mrm{b}}(\mcal{B})$,
$F_{N}(T^if):F_{N}(T^iX) \to F_{N}(T^iY)$ and $I'(F_{N}(T^if)):I'(F_{N}(T^iX)) \to I'(F_{N}(T^iY))$ are split monomorphism 
(resp., epimorphism) by Definition \ref{simpleF}.
Then \[
0 \to F_{N}(T^iX) \to F_{N}(\tau_{\geq i+1}X)\oplus I'(F_{N}(T^{-i}X^i)) \to F_{N}(\tau_{\geq i}X) \to 0
\]
is an inflation.
Let $0 \to X \to Y \to Z \to 0$ be an conflation in 
$\cat{C}^{\mrm{b}}(\Morph^{\mrm{sm}}_{N-1}(\mcal{B}))$.
Consider a small Frobenius subcategory $\mcal{C}$ of 
$\cat{C}_{N}(\Morph^{\mrm{sm}}_{N-1}(\mcal{B}))$ which contains
$F_{N}(T^{i}X), \\ F_{N}(T^{i}Y), F_{N}(T^{i}Z), I'(F_{N}(T^{-i}X), I'(F_{N}(T^{-i}Y), I'(F_{N}(T^{-i}Z),
F_{N}(\tau_{\geq i}X), \\ F_{N}(\tau_{\geq i}Y)$ and $F_{N}(\tau_{\geq i}Z)$ $(i \in \mathbb{Z})$.
Then by the induction on $j-i$,  9-lemma implies that
\[
0 \to F_{N}(\tau_{\geq i}X) \to F_{N}(\tau_{\geq i}Y) \to F_{N}(\tau_{\geq i}Z) \to 0
\]
is a short exact sequence in some abelian category.
According to Proposition \ref{excatemb},  it is a conflation in $\mcal{C}$.
Therefore, $F_{N}$ preserves conflations.
\par\noindent
Step 3. We have a functor $F_{N}: \cat{C}^{-}(\Morph^{\mrm{sm}}_{N-1}(\mcal{B})) \to \cat{C}_{N}(\mcal{B})$ which preserves 
conflations.

\par\noindent
For $X \in \cat{C}^{-}(\mcal{B})$, we have $X =\displaystyle{\lim_{i \to \infty}}\tau_{\geq -i}X$.
By the diagram \ref{limit01}, we have
\[
\tau_{\geq (\left[\frac{i+1}{2}\right])N-1}F_{N}(\tau_{\geq -i}X) =
\tau_{\geq (\left[\frac{i+1}{2}\right])N-1}F_{N}(\tau_{\geq -i-1}X)
\]
Then there exists $\displaystyle{\lim_{i \to \infty}}F_{N}(\tau_{\geq -i}X)$, and
we take $F_{N}(X)=\displaystyle{\lim_{i \to \infty}}F_{N}(\tau_{\geq -i}X)$. 
It is not hard to see that $F_{N}$ becomes a functor and 
preserves conflations.

\par\noindent
Step 4.  We have a functor $F_{N}: \cat{C}(\Morph^{\mrm{sm}}_{N-1}(\mcal{B})) \to \cat{C}_{N}(\mcal{B})$ which preserves 
conflations.

\par\noindent
Let $X:\cdots \to X^{i} \xarr{d^i} X^{i+1} \to \cdots \in \cat{C}(\mcal{B})$.
By induction on $i$, we comstruct a functor $F_{N}$.
For $X$, we have a commutative diagram in $\cat{C}^{-}(\Morph^{\mrm{sm}}_{N-1}(\mcal{B}))$:
\[\xymatrix{
0 \ar[r] & T^{-i-1}X^{i+1} \ar[r] \ar@{=}[d] & \tau_{\leq i+1}X\ar[r] \ar[d] & \tau_{\leq i}X \ar[r] \ar[d]^{d^i} & 0 \\
0 \ar[r] & T^{-i-1}X^{i+1} \ar[r] & I(T^{-i-1}X^{i+1}) \ar[r] & T^{-i}X^{i+1} \ar[r] & 0
}
\]
where all rows are exact.  Then we have a commutative diagram
in $\cat{C}_{N}(\mcal{B})$:
\begin{equation}\label{limit02}
\xymatrix{
0 \ar[r] & F_{N}(T^{-i-1}X^{i+1}) \ar[r] \ar@{=}[d] & \ar @{} [dr] |{(B)} 
\Sigma^{-1}\opn{C}(F_{N}(d^{i})) \ar[r] \ar[d] & 
F_{N}(\tau_{\leq i}X) \ar[r]\ar[d]^{F_{N}(d^i)} & 0 \\
0 \ar[r] & F_{N}(T^{-i-1}X^{i+1}) \ar[r] & I'(T^{-i-1}X^{i+1}) \ar[r] & F_{N}(T^{-i}X^{i+1}) \ar[r] & 0
}
\end{equation}
where all rows are conflations.  We take $F_{N}(\tau_{\leq i+1}X)=C(F_{N}(d^i))$.  
By the diagram \ref{limit02}, we have
\[
\tau_{\leq \left[\frac{i+2}{2}\right]N-1}F_{N}(\tau_{\leq i}X) =
\tau_{\leq \left[\frac{i+2}{2}\right]N-1}F_{N}(\tau_{\leq i+1}X)
\]
Then there exists $\displaystyle{\lim_{\infty \gets i}}F_{N}(\tau_{\leq i}X)$, and
we take $F_{N}(X)=\displaystyle{\lim_{\infty \gets i}}F_{N}(\tau_{\leq i}X)$. 
Since the commutative diagram (B) is a exact square,
it is not hard to see that $F_{N}$ becomes a functor and preserves 
conflations.

\par\noindent
Step 5.  $F_{N}$
sends projective-injective objects in $\cat{C}(\Morph^{\mrm{sm}}_{N-1}(\mcal{B}))$
to projective-\\ injective objects in $\cat{C}_{N}(\mcal{B})$.

\par\noindent
Every projective-injective object in $\cat{C}(\Morph^{\mrm{sm}}_{N-1}(\mcal{B}))$
is a direct summand of some biproduct $\oplus_{i \in \mathbb{Z}}T^i\opn{C}(1_{M_i})$
with $M_i \in \Morph^{\mrm{sm}}_{N-1}(\mcal{B})$. 
Since $F_{N}(\opn{C}(1_{M_i})) \simeq \opn{C}(1_{F_{N}(M_i)})$ is projective-injective
in $\cat{C}_{N}(\mcal{B})$, we have the statement.

\par\noindent
According to Proposition \ref{extr01}, 
$F_{N}: \cat{C}(\Morph^{\mrm{sm}}_{N-1}(\mcal{B})) \to \cat{C}_{N}(\mcal{B})$ induces
a triangle functor 
$\underline{F}_{N}: \cat{K}(\Morph^{\mrm{sm}}_{N-1}(\mcal{B})) \to \cat{K}_{N}(\mcal{B})$.
\end{proof}

\begin{lem}\label{send2Ngons00}
The following hold.
\begin{enumerate}
\item  Every complex $X$ of $\cat{C}(\Morph^{\mrm{sm}}_{N-1}(\mcal{B}))$ is of the form
\[\xymatrix{
X_{1} \ar[r]^{\alpha^1} 
& \oplus_{i=1}^{2}X_{i} \ar[r]^{\alpha^2}
&\cdots \ar[r]^{\alpha^{N-2}} & \oplus_{i=1}^{N-1}X_{i}
}\]
where $\alpha^i=
\left[\begin{smallmatrix}
1 & \cdots & 0 \\ 
 & \ddots  & \\
0 & \cdots & 1 \\
0 & \cdots & 0 \\
\end{smallmatrix}\right]$
and $(\oplus_{i=1}^{r}X_{i}, d_X^i=
\left[\begin{smallmatrix}
d^{i}_{11} & d^{i}_{12} & \cdots & d^{i}_{1r} \\ 
0 & d^{i}_{22} & \cdots & d^{i}_{2r} \\ 
\vdots  & \vdots  & \ddots &\vdots \\
0 & 0 & \cdots & d^{i}_{rr}
\end{smallmatrix}\right])$
is a complex of $\mcal{B}$.
\item  For a complex $X$ of $\cat{C}(\Morph^{\mrm{sm}}_{N-1}(\mcal{B}))$, 
the $N$-complex $F_{N}(X)$ is equal to $(Y^j, d_Y^j)$
where
\[
Y^j=\left(\oplus_{r=1}^{k}X_{r}^{2i}\right)\oplus
\left(\oplus_{r=k}^{N-1}X_{r}^{2i-1}\right) \\
\]
\[
d_Y^j=
\begin{cases}
\left[\begin{smallmatrix}
d^{i}_{11} & d^{i}_{12} & \cdots & d^{i}_{1N-1} \\ 
0 & d^{i}_{22} & \cdots & d^{i}_{2N-1} \\ 
\vdots  & \vdots  & \ddots &\vdots \\
0 & 0 & \cdots & d^{i}_{N-1N-1}
\end{smallmatrix}\right]
&\ (j\equiv -1 \mod N), 
\\
\left[\begin{matrix}
1 & 0 & \cdots & 0 &d^{i-1}_{1k}   & 0 & \cdots & 0 \\ 
0 & 1 & \ddots & \vdots & d^{i-1}_{2k}  &  \vdots & & \vdots \\ 
\vdots & \ddots & \ddots & 0 & \vdots &  \vdots  & &  \vdots  \\ 
\vdots & & \ddots & 1 &  d^{i-1}_{k-1k}  &  \vdots  & &  \vdots   \\ 
\vdots & & & \ddots & d^{i-1}_{kk} & 0 & & \vdots  \\ 
\vdots & & & & \ddots &  1 & \ddots  &  \vdots  \\ 
\vdots & & & & & \ddots   & \ddots& 0\\ 
0 & \cdots &\cdots &\cdots &\cdots &\cdots & 0 & 1\\ 
\end{matrix}\right] 
&\ \text{otherwise}
\end{cases}\]
where $i=2\left[\frac{j}{N}\right]$, $0 \leq k \leq N-1$ and $k \equiv j \mod N$.
\item  $\underline{F}_{N}(\mcal{E}^{[1, N-2]}) \subset 
\mcal{F}_{N-1}^{1}$
\end{enumerate}
\end{lem}

\begin{proof}
(1) 
It is trivial.

\par\noindent
(2)
For a complex $X \in \cat{C}^{\mrm{b}}(\Morph^{\mrm{sm}}_{N-1}(\mcal{B}))$, 
we assume $F_{N}(\tau_{\geq 2j}X)$ satisfies the statement.
For $2j=i+1$, we have the following equations in the diagram \ref{limit01}
\[\begin{aligned}
F_{N}(T^{-i}X^{i}) & =\oplus_{r=1}^{N-1}\mu_{N-r}^{N-1+jN}X^{2j-1}_r &
I'(F_{N}(T^{-i}X^{i})) & = \oplus_{r=1}^{N-1}\mu_{N-1}^{N-1+jN}X^{2j-1}_r \\
F_{N}(T^{-i-1}X^{i}) & =\oplus_{r=1}^{N-1}\mu_{r}^{r-1+jN}X^{2j-1}_r \\
\end{aligned}\]
By easy calculations, 
$
\gamma_j^{k+jN}: I'(F_{N}(T^{-i}X^{i}))^k \to Y^{k+jN}
$
is equal to
$d_Y^{k+N}$ in the statement for $0 \leq k < N$.
Therefore, we have the statement for $F_{N}(\tau_{\geq i}X)$.

\par\noindent
For $2j=i$, we have the following equations in the Diagram \ref{limit01}
\[\begin{aligned}
F_{N}(T^{-i}X^{i}) & =\oplus_{r=1}^{N-1}\mu_{r}^{r-1+(j+1)N}X^{2j}_r &
I'(F_{N}(T^{-i}X^{i})) & = \oplus_{r=1}^{N-1}\mu_{N-1}^{r-1+(j+1)N}X^{2j}_r\\
F_{N}(T^{-i}X^{i}) & =\oplus_{r=1}^{N-1}\mu_{N-r+1}^{-1+jN}X^{2j}_r \\
\end{aligned}\]
By easy calculations, 
$
\gamma_j^{k+jN}: I'(F_{N}(T^{-i}X^{i}))^k \to Y^{k+jN}
$
is a $(k+1)\times k$ matrix 
$\left[\begin{smallmatrix}
1 & \cdots & 0 \\ 
 & \ddots  & \\
0 & \cdots & 1 \\
0 & \cdots & 0 \\
\end{smallmatrix}\right]$
for $1 \leq k < N-1$, and
$
\gamma_j^{N-1+jN}: I'(F_{N}(T^{-i}X^{i}))^k \to Y^{k+jN}
$
is equal to
$d_Y^{N-1+jN}$ in the statement.
Therefore, we have the statement for $F_{N}(\tau_{\geq i}X)$.

\par\noindent
Similarly, we have the same result for the diagram \ref{limit02}.

\par\noindent
(3) By Lemma \ref{lem:comma03}, any complex of $\mcal{E}^{[1,N-2]}$
is isomorphic to the mapping cone of the following morphism
between complexes:
\[\xymatrix{
X' : \ar[d]_{f} & 0\ar[r] \ar[d] & 0 \ar[r] \ar[d]
&\cdots \ar[r] & 0 \ar[r] \ar[d] & \oplus_{i=1}^{N-1}X_{i} \ar[d]^{1} \\
X:  & X_{1} \ar[r]^{\alpha^1} 
& \oplus_{i=1}^{2}X_{i} \ar[r]^{\alpha^2}
&\cdots \ar[r]^{\alpha^{N-2}} & \oplus_{i=1}^{N-2}X_{i} \ar[r] & \oplus_{i=1}^{N-1}X_{i}
}\]
Consider a morphism between the Diagram \ref{limit01} for $F_{N}(\tau_{\geq i}X')$
and the Diagram \ref{limit01} for $F_{N}(\tau_{\geq i}X)$.
Then $F_{N}(X')=(Y'^j, d_{Y'}^i)$,
where
\[
Y^j=\begin{cases}
\oplus_{r=1}^{N-1}X_{r}^{2i}  &\ (j \equiv -1 \mod N) \\
\oplus_{r=1}^{N-1}X_{r}^{2i-1}&\ \text{otherwise}
\end{cases}
\]
\[
d_Y^j=
\begin{cases}
\left[\begin{smallmatrix}
d^{i-1}_{11} & d^{i-1}_{12} & \cdots & d^{i-1}_{1N-1} \\ 
0 & d^{i-1}_{22} & \cdots & d^{i-1}_{2N-1} \\ 
\vdots  & \vdots  & \ddots &\vdots \\
0 & 0 & \cdots & d^{i-1}_{N-1N-1}
\end{smallmatrix}\right]
&\ (j\equiv -2 \mod N), 
\\
\text{identity} &\ \text{otherwise}
\end{cases}
\]
where $i=2\left[\frac{j}{N}\right]$, $0 \leq k \leq N-1$ and $k \equiv j \mod N$.
Moreover, $F_{N}(f): F_{N}(X') \to F_{N}(X)$ is equal to $g:Y' \to Y$
where
\[g^j=
\begin{cases}
\text{identity} &\ (j\equiv 0, -1 \mod N) \\
\left[\begin{matrix}
d_{11}^{i-1} & \cdots & d_{1k}^{i-1} &  0  &  \cdots &  0  \\ 
0 & \ddots &  \vdots &  \vdots  & &  \vdots   \\ 
\vdots & \ddots & d^{i-1}_{kk} & 0 & & \vdots  \\ 
\vdots & & \ddots &  1 & \ddots  &  \vdots  \\ 
\vdots & & & \ddots   & \ddots& 0\\ 
0 &\cdots &\cdots &\cdots &0 & 1 \\ 
\end{matrix}\right] &\ \text{otherwise}
\end{cases}\]
where $i=2\left[\frac{j}{N}\right]$, $0 < k < N-1$ and $k \equiv j \mod N$.
By the proof of Theorem \ref{tstNcpx01}, we have
$Y'=\underline{I}_{0}^{\Uparrow_{1}^{N}}
\underline{J}_{0}^{\Downarrow_{1}^{N}}(Y)$ and $\epsilon_Y=g$.
Therefore $\opn{C}(F_{N}(f)) \simeq \opn{C}(g) \in \mcal{F}_{N-1}^{1}$.
\end{proof}

\begin{lem}\label{send2Ngons01}
The following hold for the above triangle functor 
$\underline{F}_{N}: \cat{K}(\Morph^{\mrm{sm}}_{N-1}(\mcal{B})) \\ \to \cat{K}_{N}(\mcal{B})$.
\begin{enumerate}
\item  $\underline{F}_{N}(\mcal{E}^{[s, N-1]}\cap\mcal{F}^{[s,N-1]}) \subset 
\mcal{F}_{0}^{s-1}\cap\mcal{F}_{s}^{N-s-1}$.
\item  $\underline{F}_{N}(\mcal{E}^{s}) \subset \mcal{F}_{s+1}^{N-2}$.
\end{enumerate}
\end{lem}

\begin{proof}
(1)
According to Lemma \ref{lem:comma03}, we may assume any complex  $Y$ of 
$\mcal{E}^{[s, N-1]}\cap\mcal{F}^{[s,N-1]}$
is of the form
\[
0 \to \cdots \to 0 \arr  Y^s \xarr{\alpha^{s}}  \cdots \xarr{\alpha^{N-2}} Y^{N-1}
\]
where $X=(X^i, d^i)=Y^s = \cdots =Y^{N-1}$ is a complex of $\mcal{B}$ and
$\alpha^{s}=\cdots =\alpha^{N-2}=1_X$.
Then by Diagrams \ref{limit01}, \ref{limit02}, 
$\underline{F}_{N}(Y)$ is isomorphic to a complex $(X'^i, d'^i)$
where
\[\begin{array}{ll}
X'^{i} = \begin{cases}
X^{2\left[\frac{i}{N}\right]-1} \ (i\equiv 0, 1, \cdots, s-1 \mod N) \\
X^{2\left[\frac{i}{N}\right]} \ (i\equiv s, s+1, \cdots, N-1 \mod N)
\end{cases} 
\\
d'^{i} = \begin{cases}
1_{X^{2\left[\frac{i}{N}\right]}} \ (i\equiv 0, 1, \cdots, s-2 \mod N) \\
d^{2\left[\frac{i}{N}\right]-1} \ (i\equiv s-1 \mod N) \\
1_{X^{2\left[\frac{i}{N}\right]}} \ (i\equiv s, s+1, \cdots, N-2 \mod N) \\
d^{2\left[\frac{i}{N}\right]} \ (i\equiv N-1 \mod N) \\
\end{cases} 
\end{array}\]
Therefore $\underline{F}_{N}(\mcal{E}^{[s, N-1]}\cap\mcal{F}^{[s,N-1]}) \subset 
\mcal{F}_{0}^{[0,s-1]}\cap\mcal{F}_{s}^{[0,N-s-1]}$.

\par\noindent
(2)
Any complex $Y$ of $\mcal{E}^{s}$ is isomorphic to the mapping cone of a morphism $\iota:Y_1 \to Y_2$
in $\cat{C}(\Morph^{\mrm{sm}}_{N-1}(\mcal{B}))$:
\[\xymatrix{
Y_1: \ar[d]^{\iota} & 0 \ar[d] \ar[r] & \cdots \ar[r] & 0 \ar[d] \ar[r] & 0 \ar[d] \ar[r] & Y^{s+1} \ar[d]^{1} \ar[r]^{\alpha^{s+1}} & \cdots \ar[r]^{\alpha^{N-2}} & Y^{N-1} \ar[d]^{1} \\
Y_2: & 0 \ar[r] & \cdots \ar[r] & 0 \ar[r] & Y^s \ar[r]^{\alpha^{s}} & Y^{s+1} \ar[r]^{\alpha^{s+1}} & \cdots \ar[r]^{\alpha^{N-2}} & Y^{N-1}
}\]
where $X=(X^i, d^i)=Y^s = \cdots =Y^{N-1}$ is a complex of $\mcal{B}$ and
$\alpha^{s}=\cdots =\alpha^{N-2}=1_X$.
By the construction of $F_{N}$ in Theorem \ref{prolong},
$\underline{F}_{N}(Y_2)=\underline{I}_{0}^{\Uparrow_{N-r}^{N}}
\underline{J}_{0}^{\Downarrow_{N-r}^{N}}(\underline{F}_{N}(Y_1)) \in \mcal{F}_{0}^{s}$ and 
$\underline{F}_{N}(\iota)=\underline{I}_{0}^{\Uparrow_{N-r}^{N}}
\underline{J}_{0}^{\Downarrow_{N-r}^{N}}(\underline{F}_{N}(Y_1)) \xarr{\varepsilon_{\underline{F}_{N}(Y_1)}} \underline{F}_{N}(Y_1)$.
By the Proof of Theorem \ref{tstNcpx01}, 
the mapping cone $\opn{C}(\underline{F}_{N}(\iota))$ is isomorphic to a complex $(X'^i, d'^i)$
where
\[\begin{array}{ll}
X'^{i} = \begin{cases}
X^{2\left[\frac{i}{N}\right]-1} \ (i\equiv 0, 1, \cdots, s-1 \mod N) \\
X^{2\left[\frac{i}{N}\right]} \ (i\equiv s \mod N) \\
X^{2\left[\frac{i}{N}\right]+1} \ (i\equiv s+1, \cdots, N-1 \mod N) 
\end{cases} 
\\
d'^{i} = \begin{cases}
1_{X^{2\left[\frac{i}{N}\right]-1}} \ (i\equiv 0, 1, \cdots, s-2 \mod N) \\
d^{2\left[\frac{i}{N}\right]-1} \ (i\equiv s-1 \mod N) \\
d^{2\left[\frac{i}{N}\right]} \ (i\equiv s \mod N) \\
1_{X^{2\left[\frac{i}{N}\right]}} \ (i\equiv s+1, \cdots, N-1 \mod N) \\
\end{cases} 
\end{array}\]
Therefore $\underline{F}_{N}(\mcal{E}^{s}) \subset \mcal{F}_{s+1}^{N-2}$.

\end{proof}

\begin{thm}\label{KNhtp}
Let $\mcal{B}$ be an additive category, then we have triangle equivalences:
\[
\cat{K}^{\sharp}(\Morph^{\mrm{sm}}_{N-1}(\mcal{B})) \simeq \cat{K}^{\sharp}_{N}(\mcal{B})
\]
where $\sharp=\text{nothing}, -, +, \mrm{b}$.
\end{thm}

\begin{proof}
We prove the statement by the following steps:
\par\noindent
Step 1. 
$\underline{F}_{N}: \cat{K}(\Morph^{\mrm{sm}}_{N-1}(\mcal{B})) \to \cat{K}_{N}(\mcal{B})$
sends \\
{\small $
(\mcal{F}^{[1,N-1]},\mcal{E}^{[2,N-1]}, \mcal{E}^{1}, \mcal{F}^{[1,2]}, \cdots,
\mcal{E}^{s}, \mcal{F}^{[s,s+1]}, \cdots, \mcal{E}^{N-2}, \mcal{F}^{[N-2,N-1]}, \mcal{E}^{N-1}, \mcal{E}^{[1,N-2]})
$}
to \\
$
(\mcal{F}_{1}^{N-2}, \mcal{F}_{0}^{1}, \mcal{F}_{2}^{N-2},  \mcal{F}_{1}^{1}, 
\cdots, \mcal{F}_{s+1}^{N-2}, \mcal{F}_{s}^{1}, 
\cdots, \mcal{F}_{N-1}^{N-2}, \mcal{F}_{N-2}^{1}, \mcal{F}_{0}^{N-2}, \mcal{F}_{N-1}^{1}).
$

\par\noindent
According to Lemma \ref{send2Ngons00} (2),
it is easy to see that 
\[\begin{array}{lll}
\underline{F}_{N}(\mcal{F}^{[1,N-1]})\subset \mcal{F}_{1}^{N-2}, &\quad
\underline{F}_{N}(\mcal{F}^{[2,N-1]}) \subset \mcal{F}_{0}^{1}, &\quad
\underline{F}_{N}(\mcal{E}^{N-1}) \subset \mcal{F}_{0}^{N-2},  \\
\underline{F}_{N}(\mcal{F}^{[s,s+1]}) \subset \mcal{F}_{s}^{1} & \quad (1\leq s <N-1)
\end{array}\]
By Lemma \ref{send2Ngons01}, we have 
$\underline{F}_{N}(\mcal{E}^{s}) \subset \mcal{F}_{s+1}^{N-2}$ $(1 \leq s < N-1)$.
By Lemma \ref{send2Ngons00} (3), we have 
$\underline{F}_{N}(\mcal{E}^{[1, N-2]}) \subset \mcal{F}_{N-1}^{1}$.

\par\noindent
Step 2. $\underline{F}_{N}$ induces a triangle equivalence
between $\mcal{F}^{[1,N-1]}$ and $\mcal{F}_{1}^{N-2}$.

\par\noindent
Consider the following diagram:
\[\xymatrix{
\cat{K}(\mcal{B})
\ar[d]_{U_{N-1}}\ar[r]^{id} & 
\cat{K}(\mcal{B}) \ar[d]^{\underline{I}_{1}^{\Uparrow_{2}^{N}}} \\
\cat{K}(\Morph_{N-1}(\mcal{B})) \ar[r]^{\quad \underline{F}_{N}} & \cat{K}_{N}(\mcal{B})
}\]
By Proposition \ref{lastpiece}, Lemma \ref{send2Ngons00}, 
it is not hard to see that we have a functorial isomorphism
$\underline{F}_{N}\circ U_{N-1} \simeq
\underline{I}_{1}^{\Uparrow_{2}^{N}}$.
By Proposition \ref{lastpiece}, $U_{N-1}$ induces a triangle equivalence between $\cat{K}(\mcal{B})$ and $\mcal{F}^{[1,N-1]}$.
On the other hand, by Corollary \ref{n-cpxgon04}
$\underline{I}_{1}^{\Uparrow_{2}^{N}}$ induces a triangle equivalence between
 $\cat{K}_{N-1}(\mcal{B})$ and $\mcal{F}_{1}^{N-2}$.
Therefore  $\underline{F}_{N}$ induces a triangle equivalence
between $\mcal{F}^{[1,N-1]}$ and $\mcal{F}_{1}^{N-2}$.

\par\noindent
Step 3. $\underline{F}_{N}$ induces a triangle equivalence
between $\mcal{E}^{[2,N-1]}$ and $\mcal{F}_{0}^{1}$.

\par\noindent
Consider the following diagram:
\[\xymatrix{
\cat{K}(\Morph_{N-2}(\mcal{B}))
\ar[d]_{E^{\Uparrow_{N-2}^{N-1}}} \ar[r]^{\quad \underline{F}_{N-1}} & 
\cat{K}_{N-1}(\mcal{B}) \ar[d]^{\underline{I}_{0}^{\Uparrow_{N-1}^{N}}} \\
\cat{K}(\Morph_{N-1}(\mcal{B})) \ar[r]^{\quad \underline{F}_{N}} & \cat{K}_{N}(\mcal{B})
}\]
Similarly, we have a functorial isomorphism
$\underline{F}_{N}\circ E^{\Uparrow_{N-2}^{N-1}} \simeq
\underline{I}_{0}^{\Uparrow_{N-1}^{N}}\circ \underline{F}_{N-1}$.
By Proposition \ref{lastpiece}, $E^{\Uparrow_{N-2}^{N-1}}$ induces
a triangle equivalence between $\cat{K}(\mcal{B})$ and $\mcal{F}^{[1,N-1]}$.
On the other hand, by Corollary \ref{n-cpxgon04}
$\underline{I}_{1}^{\Uparrow_{2}^{N}}$ induces a triangle equivalence between
 $\cat{K}_{N-1}(\mcal{B})$ and $\mcal{F}_{0}^{1}$.
By the assumption of induction on $N$, $\underline{F}_{N-1}$ is a triangle equivalence.
Therefore $\underline{F}_{N}$ induces a triangle equivalence
between $\mcal{E}^{[2,N-1]}$ and $\mcal{F}_{0}^{1}$.

\par\noindent
According to Proposition \ref{p20Apr30}, 
$\underline{F}_{N}:\cat{K}(\Morph^{\mrm{sm}}_{N-1}(\mcal{B})) \to \cat{K}_{N}(\mcal{B})$
is a triangle equivalence.
Moreover, it is easy to see the above proof is available for the case 
$\underline{F}^{\sharp}_{N}:\cat{K}^{\sharp}(\Morph^{\mrm{sm}}_{N-1}(\mcal{B})) \to \cat{K}^{\sharp}_{N}(\mcal{B})$, where $\sharp=-,+, \mrm{b}$.
\end{proof}

In \cite{IKM2}, we studied the derived category of $N$-complexes.
We have results of \cite{IKM2} Corollaries 4.11,  4.12 under the weaker condition.
 
\begin{lem}\label{qtricat01}
 Let $\mcal{D}$ be a triangulated category, $\mcal{C}, \mcal{U}$ 
 full triangulated subcategories of $\mcal{C}$.
 We assume that for any $X \in \mcal{D}$ there is a triangle 
 $C_X \to X \to U_X \to \Sigma C_X$ such that
 $C_X \in \mcal{C}$ and $U_X \in \mcal{U}$.
 Then we have a triangle equivalence
 \[
 \mcal{C}/(\mcal{C}\cap\mcal{U}) \simeq \mcal{D}/\mcal{U}.
 \]
\end{lem}

\begin{proof}
Let $E: \mcal{C} \to \mcal{D}$ be the canonical embedding, 
$Q': \mcal{C} \to \mcal{C}/(\mcal{C}\cap\mcal{U})$, 
$Q: \mcal{D} \to \mcal{D}/\mcal{U}$ the canonical quotients.
Then there is a triangle functor $F: \mcal{C}/(\mcal{C}\cap\mcal{U}) \to \mcal{D}/\mcal{U}$
such that $Q\circ E = F\circ Q'$.
By the assumption, $F$ is obviously dense.
Given $X_1, X_2 \in \mcal{C}$, any morphism in $\mcal{D}/\mcal{U}$ is represented by the following in $\mcal{D}$:
\begin{equation}\label{qfmor}
\xymatrix{
X_1 \ar[d]_{s_1}  \ar[dr]^{g}\\
F(C_1) & F(C_2)
}
\end{equation}
where $X_1 \xarr{s_1} F(C_1) \arr U_1 \arr \Sigma X_1$ is a triangle in $\mcal{D}$
with $U_1 \in \mcal{U}$.
By the assumption, there is a triangle $C_2 \xarr{s_2} X_1 \arr U_2 \arr \Sigma C_2$
with $U_2 \in \mcal{U}$.
Therefore, $F$ is a full dense functor.
Let $f: C_1 \to C_2$ be a morphism in $\mcal{C}/(\mcal{C}\cap\mcal{U})$
such that $F(f)=0$ in $\mcal{D}_\mcal{U}$.
Then $F(f)$ is represented by the diagram \ref{qfmor}
where $g=0$. 
In the above, $gs_2=0$, and then $f=0$ in $\mcal{C}/(\mcal{C}\cap\mcal{U})$.
Hence $F$ is an equivalence.
\end{proof}

We say that $\mcal{A}$ is an Ab3 category provided that
it has any coproduct of objects.
Moreover, $\mcal{A}$ is an Ab4 category provided that
it has any coproduct of objects, and that
the coproduct of monics is monic.

\begin{prop}\label{qtricat02}
Let $\mcal{A}$ be an AB4 category with enough projectives, and $\mcal{P}$ the category of projective objects, $\mcal{P}$ the full subcategory of $\mcal{A}$ consisting of projective objects.
Then the following hold.
\begin{enumerate}
\item We have a triangle equivalence
\[
\cat{K}(\Morph^{\mrm{sm}}_{N-1}(\mcal{P}))/
\cat{K}^{\phi}(\Morph^{\mrm{sm}}_{N-1}(\mcal{P}))
\simeq
\cat{D}(\Morph_{N-1}(\mcal{A})).
\]
\item  We have a triangle equivalence
\[
\cat{K}_{N}(\mcal{P})/\cat{K}^{\phi}_{N}(\mcal{P})
\simeq
\cat{D}_{N}(\mcal{A}).
\]
\end{enumerate}
Here $\cat{K}^{\phi}(\Morph^{\mrm{sm}}_{N-1}(\mcal{P}))$
(resp., $\cat{K}^{\phi}_{N}(\mcal{P})$) is the full subcategory of
Here $\cat{K}(\Morph^{\mrm{sm}}_{N-1}(\mcal{P}))$
(resp., $\cat{K}_{N}(\mcal{P})$) consisting of
complexes (resp., $N$-complexes) of which all hmologies are null.
\end{prop}

\begin{proof}
(1)
It is easy to see that 
$\Morph^{\mrm{sm}}_{N-1}(\mcal{P})$ is the full subcategory of $\Morph_{N-1}(\mcal{A})$
consisting of projective objects, and that $\Morph_{N-1}(\mcal{A})$
is an Ab4 category with enough projectives.
According to \cite{BN}, for any complex $X \in \cat{K}(\Morph_{N-1}(\mcal{A}))$, there
is a quasi-isomorphism $P \to X$ with $P \in \cat{K}(\Morph^{\mrm{sm}}_{N-1}(\mcal{P}))$.
By Lemma \ref{qtricat01}, we have 
$\cat{K}(\Morph^{\mrm{sm}}_{N-1}(\mcal{P}))/
\cat{K}^{\phi}(\Morph^{\mrm{sm}}_{N-1}(\mcal{P}))
\simeq
\cat{D}(\Morph_{N-1}(\mcal{A}))$.

\par\noindent
(2)
According to \cite{IKM2} Theorem 2.23, for any complex $X \in \cat{K}_{N}(\mcal{A})$, there
is a quasi-isomorphism $P \to X$ with $P \in \cat{K}_{N}(\mcal{P})$.
By Lemma \ref{qtricat01}, we have 
$\cat{K}_{N}(\mcal{P})/\cat{K}^{\phi}_{N}(\mcal{P})
\simeq
\cat{D}_{N}(\mcal{A})$.
\end{proof}
 
\begin{cor}\label{DNAb}
Let $\mcal{A}$ be an abelian category with enough projectives, and $\mcal{P}$ the category of projective objects.
Then we have triangle equivalences
\[\begin{aligned}
\cat{K}^{-}(\Morph^{\mrm{sm}}_{N-1}(\mcal{P})) & \simeq \cat{K}^{-}_{N}(\mcal{P}), 
\cat{K}^{\mrm{b}}(\Morph^{\mrm{sm}}_{N-1}(\mcal{P})) & \simeq \cat{K}^{\mrm{b}}_{N}(\mcal{P}),  \\
\cat{D}^{-}(\Morph_{N-1}(\mcal{A})) & \simeq \cat{D}^{-}_{N}(\mcal{A}) .
\end{aligned}\]
Moreover if $\mcal{A}$ is an Ab4 category, then
Then we have triangle equivalences:
\[
\cat{D}(\Morph_{N-1}(\mcal{A})) \simeq \cat{D}_{N}(\mcal{A}).
\]
\end{cor}

\begin{proof}
According to \cite{IKM2} Lemma 4.8, we have a triangle equivalence \\
$\cat{K}^{-}(\Morph^{\mrm{sm}}_{N-1}(\mcal{P})) \simeq \cat{D}^{-}(\Morph_{N-1}(\mcal{A}))$.
By \cite{IKM2} Theorem 2.18, we have a triangle equivalence
$\cat{K}^{-}_{N}(\mcal{P})\simeq \cat{D}^{-}_{N}(\mcal{A})$.
By Theorem \ref{KNhtp}, we have the statement.

\end{proof}

\section{Appendix}\label{appendix}

In this section, we give results concerning Frobenius categories.
Let $\mcal{C}$ be an exact category with a collection $\mcal{E}$ of exact sequences in the sense of Quillen \cite{Qu}.
An exact sequence $0 \to X \xarr{f} Y \xarr{g} Z \to 0$ in $\mcal{E}$ is called a conflation, and 
$f$ and $g$ are called an inflation and a deflation, respectively.
An additive functor $F: \mcal{C} \to \mcal{C}'$ is called exact if it sends conflations in $\mcal{C}$
to conflations in $\mcal{C}'$.
An exact category $\mcal{C}$ is called a Frobenius category provided that it has enough projectives
and enough injectives, and that any object of $\mcal{C}$ is projective if and only if it is injective.
In this case, the stable category $\underline{C}$ of $\mcal{C}$ by projective objects is a triangulated category
(see \cite{H1}).

\begin{rem}\label{extr00}
For a Frobenius category $\mcal{C}$, we treat the only case that
for any object $X$ of $\mcal{C}$ we can choose conflations
\[\begin{aligned}
0 \arr X \xarr{u_X} I_{\mcal{C}}(X) \xarr{v_X} \Sigma_{\mcal{C}} X \arr 0 \\
0 \arr \Sigma^{-1}_{\mcal{C}}X \xarr{u_{\Sigma^{-1}X}} P_{\mcal{C}}(X) \xarr{v_{\Sigma^{-1}X}} X \arr 0
\end{aligned}\]
where $I_{\mcal{C}}(X)$ and $P_{\mcal{C}}(X)$ are projective-injective objects in $\mcal{C}$.
For a morphism $f:X \to Y$ in $\mcal{C}$, we have a commutative diagram
\[\xymatrix{
0 \ar[r] & X \ar[d]^{f}\ar[r]^{u_X} & I_{\mcal{C}}(X) \ar[d]^{I_f}\ar[r]^{v_X} & \Sigma_{\mcal{C}} X \ar[d]^{\Sigma_f}\ar[r] & 0 \\
0 \ar[r] & Y \ar[r]^{u_X} & I_{\mcal{C}}(Y) \ar[r]^{v_Y} & \Sigma_{\mcal{C}} Y \ar[r] & 0 \\
}\]
It is easy to see that $\Sigma_f$ is uniquely determined in the stable category $\underline{\mcal{C}}$.
Therefore $\underline{C}$ has a suspension functor $\Sigma_{\underline{\mcal{C}}}: \underline{\mcal{C}}
\to \underline{\mcal{C}}$.
\end{rem}
x
\begin{prop}[\cite{Ke1} A.2 Proposition]\label{excatemb}
If $(\mcal{C}, \mcal{E})$ is a small exact category, there is
an equivalence $G : \mcal{C} \to \mcal{M}$ onto a full subcategory $M$ of an abelian category $\mcal{A}$ such
that $\mcal{M}$ is closed under extensions and that $\mcal{E}$ is formed by the collection of sequences
$0 \to X \xarr{f} Y \xarr{g} Z \to 0$ inducing exact sequences in $\mcal{A}$:
\[
0 \to G(X) \xarr{G(f)} G(Y) \xarr{G(g)} G(Z) \to 0
\]
\end{prop}

\begin{prop}\label{extr01}
Let $\mcal{C}, \mcal{C}'$ be Frobenius categories, $F: \mcal{C} \to \mcal{C}'$ an exact functor.
If $F$ sends projective objects in $\mcal{C}$ to projective objents in $\mcal{T}'$, then
it induces the triangle functor $\underline{F} : \underline{\mcal{C}} \to \underline{\mcal{C}}'$.
\end{prop}

\begin{proof}
For $X \in \mcal{C}$, since $F(I_{\mcal{C}}(X)),  I_{\mcal{C}'}(F(X))$ are projective-injective
objects in $\mcal{\underline{C}}'$, we have a commutative diagram
\[\xymatrix{
0 \ar[r] & F(X) \ar@{=}[d]\ar[r]^{F(u_X)} & F(I_{\mcal{C}}(X)) \ar[d]^{\theta_X}\ar[r]^{F(v_{X})} 
& F(\Sigma_{\mcal{C}} X) \ar[d]^{\eta_X}\ar[r] & 0 \\
0 \ar[r] & F(X) \ar@{=}[d]\ar[r]^{u_{F(X)}} & I_{\mcal{C}'}(F(X)) \ar[d]^{\theta'_X}\ar[r]^{v_{F(X)}} & 
\Sigma_{\mcal{C}'}F(X)\ar[d]^{\eta'_X}\ar[r] & 0 \\
0 \ar[r] & F(X) \ar[r]^{F(u_X)} & F(I_{\mcal{C}}(X)) \ar[r]^{F(v_{X})} 
& F(\Sigma_{\underline{\mcal{C}}} X) \ar[r] & 0 \\
}\]
There are $\gamma_X: F(\Sigma_{\mcal{C}} X) \to F(I_{\mcal{C}}(X))$
and $\gamma'_X: \Sigma_{\mcal{C}'}F(X) \to I_{\mcal{C}'}(F(X))$ such that 
$F(v_X)\gamma_X=1_{F(\Sigma_{\underline{\mcal{C}}} X)}-\eta'_X\eta_X$ and 
$v_{F(X)}\gamma'_X=1_{\Sigma_{\mcal{C}'}F(X)}-\eta_X\eta'_X$ in ${\mcal{C}}'$.
Then $\underline{\eta}_X$ is an isomorphism in $\underline{\mcal{C}}'$.
For a morphism $f:X \to Y$ in $\mcal{C}$, 
we have the following diagram by the diagrams  of the above and Definition \ref{extr00}:
\[
\xymatrix@!0{
& 0\ \ar@{->}[rr]
& & F(X) \ \ar@{->}[rr]^<{F(u_X)}\ar@{=}'[d][dd]
& & F(I(X))\ \ar@{->}[rr]^<{F(v_X)}\ar@{->}'[d][dd]^{\theta_X}
& & F(\Sigma X)\ \ar@{->}[rr]\ar@{->}'[d][dd]^{\eta_X}
& & 0
\\
0\ \ar@{->}[rr]
& & F(Y) \ \ar@{<-}[ur]^(.25){F(f)}\ar@{->}[rr]\ar@{=}[dd]
& & F(I(Y))\ \ar@{->}[rr]\ar@{<-}[ur]^(.25){F(I_f)}\ar@{->}[dd]^<<{\theta_Y}
& & F(\Sigma Y)\ \ar@{->}[rr]\ar@{<-}[ur]^(.25){F(\Sigma_f)}\ar@{->}[dd]^<<{\eta_Y}
& & 0\
\\
& 0\ \ar@{->}'[r][rr]
& & F(X)\ \ar@{->}'[r][rr]
& & I(F(X))\ \ar@{->}'[r][rr]
& & \Sigma F(X)\ \ar@{->}[rr]
& & 0\
\\
0\ \ar@{->}[rr]
& & F(Y)\ \ar@{->}[rr]_>>{u_{F(Y)}}\ar@{<-}[ur]_(.8){F(f)}
& & I(F(Y))\ \ar@{->}[rr]_>{v_{F(Y)}}\ar@{<-}[ur]_(.8){I_{F(f)}}
& & \Sigma F(Y)\ \ar@{->}[rr]\ar@{<-}[ur]_(.8){\Sigma_{F(f)}}
& & 0\ \ar@{}[ur]}
\]
where the diagrams on the top, the bottom, the front and the back surfaces are commutative.
Since
$
(\theta_YF(I_f) -I_{F(f)}\theta_X)F(u_X) =u_{F(Y)}F(f) -I_{F(f)}u_{F(X)} =0
$
there is a morphism $\delta:F(\Sigma X) \to F(I(Y))$ such that
$\theta_YF(I_f) -I_{F(f)}\theta_X=\delta F(v_X)$.
Then we have $\eta_YF(\Sigma_f)-\Sigma_{F(f)}\eta_X=v_{F(Y)}\delta$, and then
$\underline{\eta}_Y\underline{F}(\Sigma(f))=\Sigma(\underline{F}(f)) \underline{\eta}_X$ in $\mcal{C}'$.
Therefore we have a functorial isomorphism 
$\underline{\eta}: \underline{F}\circ\Sigma_{\underline{\mcal{C}}} \iso 
\Sigma_{\underline{\mcal{C}}'}\circ\underline{F}$.
For a triangle $X \xarr{\underline{f}} Y  \xarr{\underline{g}} Z  \xarr{\underline{h}} \Sigma X$
in $\underline{\mcal{C}}$, 
we may have a morphism between conflations:
\[\xymatrix{
0 \ar[r] & X \ar[d]^{f}\ar[r]^{u_X} & I(X) \ar[d]^{\psi_f}\ar[r]^{v_X} & \Sigma X \ar@{=}[d]\ar[r] & 0 \\
0 \ar[r] & Y \ar[r]^{g} & Z \ar[r]^{h} & \Sigma X \ar[r] & 0 
}\]
There are morphisms $\psi_{F(f)}:I_{\mcal{C}'}(F(X)) \arr Z'$ and 
$\xymatrix{z:F(Z) \ar@{-->}[r] & Z'}$ such that we have a commutative diagram
\[
\xymatrix@!0{
& 0\ \ar@{->}[rr]
& & F(X) \ \ar@{->}[rr]^<{F(u_X)}\ar@{=}'[d][dd]
& & F(I(X))\ \ar@{->}[rr]^<{F(v_X)}\ar@{->}'[d][dd]^{\theta_X}
& & F(\Sigma X)\ \ar@{->}[rr]\ar@{->}'[d][dd]^{\eta_X}
& & 0
\\
0\ \ar@{->}[rr]
& & F(Y) \ \ar@{<-}[ur]^(.25){F(f)}\ar@{->}[rr]\ar@{=}[dd]
& & F(Z)\ \ar@{->}[rr]\ar@{<-}[ur]^(.25){F(\psi_{f})}\ar@{-->}[dd]^<<{z}
& & F(\Sigma X)\ \ar@{->}[rr]\ar@{=}[ur]\ar@{->}[dd]^<<{\eta_X}
& & 0\
\\
& 0\ \ar@{->}'[r][rr]
& & F(X)\ \ar@{->}'[r][rr]
& & I(F(X))\ \ar@{->}'[r][rr]
& & \Sigma F(X)\ \ar@{->}[rr]
& & 0\
\\
0\ \ar@{->}[rr]
& & F(Y)\ \ar@{->}[rr]\ar@{<-}[ur]_(.8){F(f)}
& & Z' \ \ar@{->}[rr]\ar@{<-}[ur]_(.8){\psi_{F(f)}}
& & \Sigma F(X)\ \ar@{->}[rr]\ar@{=}[ur]
& & 0\ \ar@{}[ur]}
\]
Similarly, there is a morphism $z':Z' \to F(Z)$ such that
we have  the above commutative diagram of which all vertical arrows are reversed.
Put $\zeta=z'z+F(\psi_{f})\gamma_XF(h)$, we have a commutative diagram
\[\xymatrix{
0 \ar[r] & F(Y) \ar@{=}[d] \ar[r]^{F(g)} & F(Z) \ar[d]^{\zeta} \ar[r]^{F(h)} & F(\Sigma X) \ar@{=}[d]\ar[r] & 0 \\
0 \ar[r] & F(Y) \ar[r]^{F(g)} & F(Z) \ar[r]^{F(h)} & F(\Sigma X) \ar[r] & 0
}\]
Then $\underline{z}'\underline{z}$ is an isomorphism in $\underline{\mcal{C}}'$.
Similarly $\underline{z}\underline{z}'$ is also an isomorphism in $\underline{\mcal{C}}'$, and 
then $\underline{z}$ is an isomorphism in $\underline{\mcal{C}}'$.
Therefore we have a commutative diagram in $\underline{\mcal{C}}'$:
\[\xymatrix{
\underline{F}(X) \ar@{=}[d]\ar[r]^{\underline{F}(\underline{f})} & \underline{F}(Y) \ar@{=}[d]\ar[r]^{\underline{F}(\underline{g})} 
& F(Z) \ar[d]^{\underline{z}}\ar[r]^{\underline{F}(\underline{h})} & \underline{F}(\Sigma X) \ar[r]^{\underline{\eta}_X}  & \Sigma \underline{F}(X)\ar@{=}[d]\\
\underline{F}(X) \ar[r]^{\underline{F}(\underline{f})} & \underline{F}(Y) \ar[r]^{\underline{g}'} & Z' \ar[rr]^{\underline{h}'} & & \Sigma \underline{F}(X)
}\]
where all vertical arrows are isomorphisms.
Hence $F$ induces a triangle functor $\underline{F}:\underline{\mcal{C}} \to \underline{\mcal{C}}'$.
\end{proof}

\begin{prop}\label{triadjoint02}
Let $\mcal{C}, \mcal{C}'$ be Frobenius categories, $F: \mcal{C} \to \mcal{C}'$, $G: \mcal{C}' \to \mcal{C}$ 
exact functors such that $F$ is a right adjoint of $G$.
Then $F$ and $G$ induce triangle functors
$\underline{F}: \underline{\mcal{C}} \to \underline{\mcal{C}}'$, 
$\underline{G}: \underline{\mcal{C}}' \to \underline{\mcal{C}}$ such that
$\underline{F}$ is a right adjoint of $\underline{G}$.
\end{prop}

\begin{proof}
Let $0 \to A \to B \to C \to 0$ be a conflation in $\mcal{C}$, then
$0 \to F(A) \to F(B) \to F(C) \to 0$ is a conflation in $\mcal{C}'$.
For a projective-injective object $P$ of $\mcal{C}'$, 
we have an isomorphism between exact sequences:
\[\xymatrix{
0 \ar[r] & \Hom_{\mcal{C}}(P, F(A)) \ar[d]^{\wr}\ar[r] & \Hom_{\mcal{C}}(P, F(B)) \ar[d]^{\wr}\ar[r]
& \Hom_{\mcal{C}}(P, F(C)) \ar[d]^{\wr}\ar[r] & 0 \\
0 \ar[r] & \Hom_{\mcal{C}}(G(P), A) \ar[r] & \Hom_{\mcal{C}}(G(P), B) \ar[r]
& \Hom_{\mcal{C}}(G(P), C)
}\]
Then $G(P)$ is a projective-injective object in $\mcal{C}$.
Similarly, if $Q$ is a projective-injective object of $\mcal{C}$, then
$F(Q)$ is also a projective-injective object of $\mcal{C}'$.
By Proposition \ref{extr01}, $F$ and $G$ induce the triangle functors
$\underline{F}: \underline{\mcal{C}} \to \underline{\mcal{C}}'$, 
$\underline{G}: \underline{\mcal{C}}' \to \underline{\mcal{C}}$.
Given $X \in \mcal{C}'$, consider a conflation
\[
0 \arr X \arr I(X) \arr \Sigma X \arr 0
\]
For any $Y \in \mcal{C}$, we have a isomorphism between exact sequences
\[\xymatrix{
\Hom_{\mcal{C}}(I(X),F(Y)) \ar[d]^{\wr} \ar[r] 
& \Hom_{\mcal{C}}(X,F(Y)) \ar[d]^{\wr} \ar[r] 
& \Hom_{\underline{\mcal{C}}}(X,F(Y)) \ar[d] \ar[r] 
& 0 \\
\Hom_{\mcal{C}}(G(I(X)),Y) \ar[r] 
& \Hom_{\mcal{C}}(G(X),Y) \ar[r] 
& \Hom_{\underline{\mcal{C}}}(G(X),Y) \ar[r] 
& 0 \\
}\]
Then we have an isomorphism 
$\Hom_{\underline{\mcal{C}}}(X, F(-)) \simeq \Hom_{\underline{\mcal{C}}}(G(X),-)$.
Similarly we have an isomorphism 
$\Hom_{\underline{\mcal{C}}}(-, F(Y)) \simeq \Hom_{\underline{\mcal{C}}}(G(-),Y)$.
\end{proof}


\end{document}